\documentclass{amsart}
\usepackage{amscd, amsmath, amssymb, amsfonts, epic, amsthm, hyperref, latexsym}
\usepackage{graphicx}
\usepackage{a4wide}
\usepackage{xypic}
\usepackage{paralist}
\usepackage{booktabs}
\usepackage{multirow}
\usepackage{epsfig}
\usepackage{graphics}

\def\P{\text{\bf P}}

\def\arrow{\rightarrow}

\def\iso{\cong}

\def\Pic{\text{Pic}} 

\def\Aut{\text{Aut}}

\newcommand{\PP}{\mathbb{P}}
\newcommand{\ZZ}{\mathbb{Z}}

\newcommand{\OO}{\mathcal{O}}
\newcommand{\QQ}{\mathbb{Q}}
\newcommand{\RR}{\mathbb{R}}

\def\PsAut{\text{PsAut}}

\newcommand{\Nef}[1]{\overline{A(#1)}}
\newcommand{\Mov}[1]{\overline{M(#1)}}
\newcommand{\Curv}[1]{\overline{\text{Curv}(#1)}}

\newtheorem{theorem}{Theorem}[section]
\newtheorem{lemma}[theorem]{Lemma}
\newtheorem{corollary}[theorem]{Corollary} 
\newtheorem{proposition}[theorem]{Proposition} 
 
\newtheorem{conjecture}[theorem]{Conjecture} 

\newtheorem{claim}[theorem]{Claim}

\theoremstyle{remark}
\newtheorem{remark}[theorem]{Remark}

\begin{document}
\Large
\title{Fano manifolds of index $n-1$ and the cone conjecture}

\author{Izzet Coskun} 
\author{Artie Prendergast-Smith}
\date{} 

\address{University of Illinois at Chicago, Department of Mathematics,
  Statistics and Computer Science, Chicago, IL 60607}

\email{coskun@math.uic.edu, artie@math.uic.edu}

\thanks{During the preparation of this article the first author was
  partially supported by the NSF CAREER grant DMS-0950951535, and an
  Alfred P. Sloan Foundation Fellowship.}

\begin{abstract}
  The Morrison-Kawamata cone conjecture predicts that the actions of
  the automorphism group on the effective nef cone and the
  pseudo-automorphism group on the effective movable cone of a klt
  Calabi-Yau pair $(X, \Delta)$ have finite, rational polyhedral
  fundamental domains.  Let $Z$ be an $n$-dimensional Fano manifold of
  index $n-1$ such that $-K_Z = (n-1) H$ for an ample divisor
  $H$. Let $\Gamma$ be the base locus of a general $(n-1)$-dimensional
  linear system $V \subset |H|$. In this paper, we verify the
  Morrison-Kawamata cone conjecture for the blow-up of $Z$ along
  $\Gamma$.
\end{abstract}

\maketitle
\tableofcontents

\section{Introduction}\label{section-introduction}

Let $n \geq 3$ be a positive integer. Let $Z$ be an $n$-dimensional
Fano manifold of index $n-1$ over an algebraically closed field $k$ of
characteristic zero. Let $H$ be an ample divisor that satisfies
$-K_Z = (n-1) H$.  Let $V \subset | H |$ be a general
$(n-1)$-dimensional linear system with base locus $\Gamma$ and let $X$
be the blow-up of $Z$ along $\Gamma$. In this paper, we verify the
Morrison-Kawamata cone conjecture for $X$.  \smallskip

The nef and movable cones are among the most important invariants of a
projective variety $Y$. Recall that a divisor $D$ is nef if $D \cdot C
\geq 0$ for every curve $C$. A divisor $D$ is movable if its stable
base locus has codimension at least two in $Y$. Nef and movable
divisors form two convex cones in the N\'{e}ron-Severi space $N^1(Y)$
of $Y$, which control the contractions and the birational contractions
of $Y$, respectively.  \smallskip

For Fano varieties, these cones are as simple as possible. The nef
cone is rational polyhedral by the Cone Theorem of
Mori--Kawamata--Koll\'ar--Reid--Shokurov \cite[Theorem
3.7]{KollarMori1998}, and the movable and effective cones are rational
polyhedral by a theorem of Birkar--Cascini--Hacon--McKernan
\cite{BCHM}. However, once we leave the world of Fano varieties, the
cones may be very complicated.  Even for simple examples of
Calabi--Yau varieties such as $K3$ surfaces, the nef cone can have
uncountably many extremal rays.  \smallskip

Nevertheless, one can still hope for a satisfying description of these
cones in the Calabi--Yau case if automorphisms are taken into
account. The Morrison--Kawamata cone conjecture \cite{Morrison1992,
  Kawamata1997} predicts that for a Calabi--Yau variety, there are
rational polyhedral fundamental domains both for the action of
automorphisms on the effective nef cone and for the action of certain
birational automorphisms on the effective movable cone. We now recall
the statement of the Morrison--Kawamata cone conjecture that we will
need in this paper.  We refer the reader to Section 1 of
\cite{Totaro2008b} for a generalization of the conjecture to the
relative setting and for history and examples.  \medskip

\noindent {\bf The Morrison--Kawamata Cone Conjecture.} Throughout
this paper, a {\it rational polyhedral cone} in a real vector space
with a $\QQ$-structure means a closed convex cone with finitely many
extremal rays, each spanned by a rational vector. For $X$ a normal
projective variety, let $N^1(X)$ denote the N\'{e}ron-Severi space of
Cartier divisors on $X$ modulo numerical equivalence. We denote by
$N^1(X)_\ZZ$ the free abelian group in $N^1(X)$ consisting of
numerical classes of Cartier divisors. We denote the nef cone and the
closure of the movable cone of $X$ by $\Nef{X}$ and $\Mov{X}$,
respectively. Let $B^e(X)$ denote the cone generated by effective
Cartier divisors. We denote by $\Nef{X}^e$ and $\Mov{X}^e$ the
intersections $\Nef{X} \cap B^e(X)$ and $\Mov{X} \cap B^e(X)$, and
call them the {\it effective nef cone} and the {\it effective movable
  cone}, respectively.  \smallskip

Define a {\it pseudo-isomorphism} from $X_1$ to $X_2$ to be a
birational map $X_1 \dashrightarrow X_2$ which is an isomorphism in
codimension one. A {\it small $\QQ$-factorial modification} (SQM) of
$X$ is a pseudo-isomorphism from $X$ to another $\QQ$-factorial
projective variety. An SQM $\alpha: X \dashrightarrow X'$ gives an
identification of the vector spaces $N^1(X)$ and $N^1(X')$ by
pushforward and pullback. Since the pullback of an ample divisor is
movable, this identification embeds $\Nef{X'}^e$ as a subcone of
$\Mov{X}^e$.  This embedding depends on the map $\alpha: X
\dashrightarrow X'$, so we denote the image by $\Nef{X',\alpha}^e$.
Pre-composing the SQM $\alpha$ with a pseudo-automorphism of $X$ gives
another SQM $\beta: X \dashrightarrow X'$, and hence another embedding
of $\Nef{X'}^e$ in $\Mov{X}^e$. We can sum this up by saying: {\it
  pseudo-automorphisms of $X$ permute the nef cones
  $\Nef{X',\alpha}^e$ of SQMs of $X$ inside $\Mov{X}^e$}.  \smallskip

For an $\RR$-divisor $\Delta$ on a normal $\QQ$-factorial variety $X$,
the pair $(X,\Delta)$ is {\it klt} if, for all resolutions $\pi:
\tilde{X} \arrow X$ with a simple normal crossing $\RR$-divisor
$\tilde{\Delta}$ such that $K_{\tilde{X}}+\tilde{\Delta} =
\pi^*(K_X+\Delta)$, the coefficients of $\tilde{\Delta}$ are less than
1. In particular, if $X$ is smooth and $D$ is a smooth divisor on $X$,
then $(X,rD)$ is klt for any $r<1$. We say that $(X,\Delta)$ is a
{\it klt Calabi--Yau pair} if $(X,\Delta)$ is a $\QQ$-factorial klt
pair with $\Delta$ effective such that $K_X+\Delta$ is numerically
trivial. 
\smallskip

We denote the groups of automorphisms or pseudo-automorphisms of $X$
which preserve a divisor $\Delta$ by $\Aut(X,\Delta)$ and
$\PsAut(X,\Delta)$, respectively. Note that the action of
$\Aut(X,\Delta)$ and $\PsAut(X,\Delta)$ on $N^1(X)$ is determined by
the images of the representations $\Aut(X,\Delta) \arrow
\operatorname{GL}(N^1(X)_\ZZ)$ and $\PsAut(X,\Delta) \arrow
\operatorname{GL}(N^1(X)_\ZZ)$. We denote the images of these
representations by $\Aut^*(X,\Delta)$ and $\PsAut^*(X,\Delta)$. We say
that a finite rational polyhedral cone $\Pi$ is a {\em fundamental
  domain} for the action of a group $G$ on a cone $C$ if $C = G \cdot
\Pi$ and the interiors of $\Pi$ and $g \Pi$ are disjoint for $1 \not=
g \in G$. With these conventions in place, we are ready to state the
Morrison-Kawamata cone conjecture.

\begin{conjecture}[Morrison--Kawamata] \label{conj-coneconjecture}
Let $(X, \Delta)$ be a klt Calabi--Yau pair. Then:

\begin{enumerate}
\item The number of $\Aut(X,\Delta)$-equivalence classes of faces of
the effective nef cone $\Nef{X}^e$ corresponding to birational
contractions or fiber space structures is finite. Moreover, there
exists a finite rational polyhedral cone $\Pi$ which is a fundamental
domain for the action of $\Aut^*(X,\Delta)$ on $\Nef{X}^e$. 
\smallskip

\item The number of $\PsAut(X,\Delta)$-equivalence classes of
nef cones $\Nef{X', \alpha}^e$ in the cone $\Mov{X}^e$
corresponding to marked SQMs  with
marking $\alpha: X' \dashrightarrow X$ is finite. Moreover, there
exists a finite rational polyhedral cone $\Pi'$ which is a fundamental
domain for the action of $\PsAut^*(X,\Delta)$ on $\Mov{X}^e$.
\end{enumerate}
\end{conjecture}

The conjecture has been proved for Calabi--Yau surfaces by
Looijenga--Sterk and Namikawa \cite{Sterk1985,Namikawa1985}, for klt
Calabi--Yau pairs of dimension 2 by Totaro \cite{Totaro2008b}, for
Calabi--Yau fiber spaces of dimension 3 over a positive-dimensional
base by Kawamata \cite{Kawamata1997}, and for abelian varieties by the
second author \cite{Prendergast-Smith2010}. Moreover, the first
statement of the movable cone conjecture for hyperk\"ahler varities
was proved by Markman \cite{Markman}. For Calabi--Yau 3-folds
there are significant results by Oguiso--Peternell \cite{Oguiso2001},
Szendr\"oi \cite{Szendroi1999}, Uehara \cite{Uehara2004}, and Wilson
\cite{Wilson1992}, and verifications of several special cases
\cite{Borcea, GrassiMorrison, Fryers, Oguiso2011, Oguiso2012, LazicPeternell},
but the full conjecture remains open.  \smallskip

In this paper, we verify the Morrison--Kawamata conjecture for certain
blowups of Fano manifolds of index $n-1$. We now explain our examples
more precisely.  Fano manifolds of index $n-1$ are called {\em del
  Pezzo manifolds} and were classified by Fujita in \cite{Fujita1} and
\cite{Fujita2}.  If $Z$ is a del Pezzo manifold of dimension $n \geq
3$, then $Z$ is one of the following:
\begin{enumerate}
\item $Z$ is a linear section of the Grassmannian $Gr(2,5)$ in its
  Pl\"{u}cker embedding.
\item $Z = Q_1 \cap Q_2 \subset \PP^{n+2}$ is an intersection of two
  quadric hypersurfaces in $\PP^{n+2}$.
\item $Z$ is a cubic hypersurface in $\PP^{n+1}$.
\item $Z = \PP^2 \times \PP^2$.
\item $Z = F(1,2;3)$ is the flag variety parameterizing full flags in $k^3$.
\item $Z = \PP^1 \times \PP^1 \times \PP^1$. 
\item $Z \rightarrow \PP^n$ is a double cover of $\PP^n$ branched
  along a quartic hypersurface.
\item $Z$ is the blow-up of $\PP^3$ at a point.
\item $Z$ is a sextic hypersurface in the weighted projective space
  $\PP(3,2,1, \dots, 1)$.
\end{enumerate}

\smallskip
Let $-K_Z = (n-1) H$ and let $V \subset |H|$ be a general linear
system of dimension $n-1$. Then $V$ has $H^n=r$ base points $\Gamma =
\{ p_1, \dots, p_r \}$. Let $[r]$ denote the set of integers $\{1,
\dots, r\}$. Let $\pi: X \rightarrow Z$ be the blowup of $Z$ along
$\Gamma$ and let $E_i$ denote the exceptional divisor over $p_i$.  In
this paper, we verify Conjecture \ref{conj-coneconjecture} for $X$.
\smallskip

Note that if we are in case (8), so $Z$ is the blowup of $\PP^3$ at a
point, then $H= 2\OO(1)-E$ is the linear system of quadrics through
the point, and so $X$ is the blowup of $\PP^3$ in the base locus of a
general net of quadrics. The conjecture has already been verified in
this case in \cite{artie}. Also, in case (9) the conjecture turns out
to be somewhat trivial; we give the details of the proof in Section
\ref{appendix}. For the rest of the paper, therefore, we will
concentrate on cases (1)--(7). To simplify notation, we will refer to
case (1) throughout the paper as $Gr(2,5)$, but the reader should bear
in mind that our argument applies equally well to a linear section of
this variety of dimension at least 3.

\smallskip

Let us explain how the cone conjecture applies to our examples. Let
$X$ be a smooth variety such that $-K_X$ is semi-ample. Choose a
sufficiently large integer $m$ so that the line bundle $-mK_X$ is
base-point-free. By Bertini's Theorem, a general divisor $D$ in the
linear system $|-mK_X|$ is smooth. If we define $\Delta$ to be the
$\QQ$-divisor $\Delta=\frac{1}{m}D$, then $K_X + \Delta$ is
numerically trivial, and $(X,\Delta)$ is a klt pair. So in this case
the cone conjecture makes predictions about the action of
automorphisms or pseudo-automorphisms on the nef cone $\Nef{X}^e$ or
movable cone $\Mov{X}^e$.

\smallskip

Given a klt Calabi-Yau pair, the fundamental difficulty in proving the
cone conjecture is finding the necessary automorphisms and
pseudo-automorphisms needed to verify the conjecture. In Theorem
\ref{nef}, we will show that, in our examples, the nef cone is equal
to the effective nef cone and is rational polyhedral. As a
consequence, we will deduce part (1) of Conjecture
\ref{conj-coneconjecture}. Part (2) of Conjecture
\ref{conj-coneconjecture} is considerably more difficult.  By abuse of
notation, denote the pullback of $H$ by the blowup map $\pi$ also by
$H$. Then $-\frac{1}{n-1} K_X = H - \sum_{i \in [r]} E_i$ defines a
morphism $f: X \rightarrow \PP^{n-1}$ whose general fiber is an
elliptic curve. The elliptic fibration comes with $r$ sections given
by the base-points. This elliptic fibration structure with $r$
sections provides the pseudo-automorphisms that we need to verify the
Morrison--Kawamata cone conjecture. The proof involves some beautiful
explicit classical geometry.  \smallskip


The organization of the paper is as follows. In \S
\ref{section-grassmannian}, we will describe the varieties $X$ and the
elliptic fibrations $f: X \rightarrow \PP^{n-1}$ in greater detail. In
particular, we will show that the fibers of $f$ that are reducible
occur in codimension two. In \S \ref{section-nef}, we will compute the
nef cone of $X$ and prove the first part of Conjecture
\ref{conj-coneconjecture} for our examples. In \S \ref{section-flops},
we will collect basic facts concerning certain flops of $X$ that will
allow us to construct a fundamental domain for the action of the group
of pseudo-automorphisms. In \S \ref{section-movablecone}, we
will study the action of the group of pseudo-automorphisms on the
movable cone in detail to prove the second part of Conjecture
\ref{conj-coneconjecture}.  Finally, in \ref{appendix}, we will explain the case (5). \smallskip

\noindent {\bf Acknowledgments:} We thank Burt Totaro for fruitful
discussions. The first author would like to thank Donghoon David Hyeon
and POSTECH for their hospitality while part of this work was
completed.

\section{The elliptic fibration} \label{section-grassmannian}

In this section, we study the geometry of the elliptic fibration $f: X
\rightarrow \PP^{n-1}$. We discuss the locus of reducible fibers in
detail and show that it has codimension two.  \smallskip

Let $Z$ be a Fano manifold of dimension $n$ and index $n-1$ such that
$-K_Z = (n-1) H$ for a very ample divisor $H$. Consider the embedding
of $Z$ in $\PP^N$ by the complete linear system $|H|$. Let $V$ be a
general $(n-1)$-dimensional linear system of hyperplanes, and consider
the corresponding rational map $\tilde{f}: Z \dashrightarrow V^* =
\PP^{n-1}$. The base locus of $V$ in $\PP^N$ is a linear space
$\Lambda$ of dimension $\PP^{N-n}$. The base locus of $V$ restricted
to $Z$ is a union of $r$ points $\Gamma = \{ p_1, \dots, p_r\}$, where
$r = H^n$. Let $$g: X = \operatorname{Bl}_{\Gamma} Z \rightarrow Z$$
be the blowup of $Z$ along $\Gamma$. Let $E_i$ denote the exceptional
divisor over $p_i$. We then get a morphism $$f: X \rightarrow
\PP^{n-1}$$ given by the sections of $$g^* H - E_1 - \cdots - E_r = -
\frac{1}{n-1} K_X.$$ In particular, $-K_X$ is semi-ample, so the cone
conjecture applies to $X$ as explained in \S \ref{section-introduction}.
\smallskip

By adjunction, the smooth fibers of $f$ are curves with trivial
canonical bundle, therefore, $f$ is an elliptic fibration. The
exceptional divisors $E_1, \dots, E_r$ give sections of the fibration,
hence restrict to $k(\PP^{n-1})$-rational points on the generic fiber
$X_{\nu}$ of $f$.  \smallskip

We now analyze the fibers of $f$ in some detail. We first note that
$f$ is flat. By \cite[Theorem 23.1]{Matsumura1989}, $f$ is flat if and
only if the fibers of $f$ are equi-dimensional.  Each fiber of $f$ is
the proper transform in $X$ of the intersection of $Z$ with a linear
space $S \cong \PP^{N-n+1}$ containing $\Lambda$. If $\dim(S \cap Z)>
1$, then $\dim(\Lambda \cap Z) = \dim(S \cap Z) -1 \geq 1$. This
contradicts that $\Lambda \cap Z$ is the finite set of points
$\Gamma$. We conclude that $f$ is flat and every fiber of $f$ has
dimension one. We already observed that the general fiber of $f$ is a
smooth genus one curve. However, $f$ may have reducible and singular
fibers.

\begin{lemma}\label{d-points}
  Let $C$ be an irreducible component of the fiber of $f$ of degree
  $d< r$. Then $g(C)$ contains exactly $d$ of the points $\Gamma = \{p_1,
  \dots, p_r\}$.
\end{lemma}

\begin{proof}
  The fibers of $f$ are curves of degree $r$.  Suppose that $C$
  contains $d+1$ points $p_{i_1}, \dots, p_{i_{d+1}}$ of
  $\Gamma$. Then by Bezout's Theorem, $C$ is contained in the linear
  span of the points $p_{i_1}, \dots, p_{i_{d+1}}$. Since the linear
  span of the points $p_1, \dots, p_r$ intersects $Z$ in finitely many
  points, we obtain a contradiction. We conclude that $g(C)$ can
  contain at most $d$ points. However, since the image of the fiber
  under $g$ contains all the points $p_1, \dots, p_r$ and the residual
  curve (which has degree $r-d$) can only contain at most $r-d$
  points, we conclude that $g(C)$ contains exactly $d$ of the points.
\end{proof}

We now run through the list of possible $Z$ and describe the geometry
in each case.  \medskip

\noindent {\bf (1) $Z$ is the Grassmannian $Gr(2,5)$.} Let $Gr(2,5)$
denote the Grassmannian parameterizing two-dimensional subspaces of $W
\cong k^5$. For $3 \geq a \geq b \geq 0$, let $\Sigma_{a,b}(F_{n-1-a}
\subset F_{n-b})$, where $F_i$ is an $i$-dimensional subspace of $W$,
denote the Schubert variety parameterizing two-dimensional subspaces
of $F_{n-b}$ that intersect $F_{n-1-a}$ non-trivially. We denote the
class of the Schubert variety by $\sigma_{a,b}$. When $b=0$, we will
abuse notation and denote the Schubert variety by
$\Sigma_a(F_{n-1-a})$. The codimension of the Schubert variety with
class $\sigma_{a,b}$ is equal to $a+b$.  \smallskip

The Pl\"{u}cker map embeds $i: Gr(2,5) \hookrightarrow \PP(\bigwedge^2
W) \cong \PP^9$. The pullback of the hyperplane class via the
Pl\"{u}cker embedding is the Schubert divisor with class
$\sigma_1$. The canonical class is given by $K_{Gr(2,5)} = - 5
\sigma_1= - 5 \ i^* (H)$. The degree of $Gr(2,5)$ is given by
$\sigma_1^6 = 5$, which can be seen by repeated applications of
Pieri's rule \cite[I.5]{gh}. The linear system $V \subset |H|$ is the
set of hyperplanes that contain a general $\Lambda \cong \PP^3 \subset
\PP^9$. Hence, $X$ is the blowup of $Gr(2,5)$ at the five points
$\Gamma = \{ p_1, \dots, p_5\} = Gr(2,5) \cap \Lambda$. The fibers of
$f : X \rightarrow \PP^5$ are elliptic normal curves of degree $5$ in
$\PP^4$. Hence, the reducible fibers contain an irreducible component
of degree one or two.  \smallskip

A line in the Grassmannian $Gr(2,5)$ parameterizes a pencil $L_t$ of
two-dimensional subspaces that contain a fixed vector $p$ and are
contained in a fixed three-dimensional subspace $S$ of $W$. In
particular, there exists a line in $Gr(2,5)$ through every point. Let
$\Omega_i$ be the two-dimensional subspace corresponding to the point
$p_i \in Gr(2,5)$. The span of a line $l$ through $p_i$ and the
remaining points of $\Gamma$ is a $\PP^4$ that intersects $Gr(2,5)$ in
a degree $5$ reducible curve containing $l$ as a component.
\smallskip

Conversely, suppose that $f$ has a reducible fiber containing a line
$l$. By Lemma \ref{d-points}, $l$ has to contain one of the points of
$\Gamma$. Assume without loss of generality that $l$ contains
$p_1$. The linear spaces parameterized by $l$ intersect $\Omega_1$
non-trivially. Consequently, every line passing through $p_1$ is
contained in the Schubert variety $\Sigma_2(\Omega_1)$. Furthermore,
every point of $\Sigma_2(\Omega_1)$ contains a line passing through
$p_1$.  The Schubert variety $\Sigma_2(\Omega_1)$ is a cone over the
Segre variety $\PP^1 \times \PP^2$ \cite{Coskun2011}. Restricted to
this locus, the map $f$ is the projection from the cone point. We
conclude that the fiber contains a line passing through $p_1$
precisely over fibers contained in a Segre embedding of $\PP^1 \times
\PP^2$ in $\PP^5$. In particular, this locus has codimension two.
\smallskip

We can view $Gr(2,5)$ also as $\mathbb{G}(1,4)$ parameterizing
projective lines in $\PP^4$. The lines parameterized by a conic in
$\mathbb{G}(1,4)$ sweep out a quadric surface in $\PP^4$.  If the
plane spanned by the conic is contained in $\mathbb{G}(1, \PP V)$,
then the quadric surface is singular or a double plane, depending on
the cohomology class of the plane. If the cohomology class of the
plane is $\sigma_{3,1}$, then the quadric surface is a cone. If the
cohomology class of the plane is $\sigma_{2,2}$, then the quadric
surface is a double plane. If the plane of the conic is not contained
in the Grassmannian, then the quadric surface is smooth. In
particular, if two points on the conic correspond to two skew-lines,
then the quadric surface is smooth.  Given two points $p,q$ in the
$\mathbb{G}(1, \PP V)$, the corresponding lines $L_p$ and $L_q$ in
$\PP^4$ span a $\PP^3$ or $\PP^2$. Since given any three lines in
$\PP^3$, there exists a (possibly reducible) quadric surface
containing them, we conclude that there is a conic through any two
points of $Gr(2,5)$.  \smallskip

Suppose that a fiber of $f$ contains an irreducible conic $C$ passing
through $p_1$ and $p_2$. Since the line joining $p_1$ and $p_2$ is not
contained in $Gr(2,5)$, we conclude that the plane spanned by $C$
cannot be contained in the Grassmannian. Thus the lines parameterized
by $C$ sweep out a smooth quadric surface. The lines $\PP \Omega_1$
and $\PP \Omega_2$ corresponding to the points $p_1$ and $p_2$ span a
$\PP^3$ in $\PP W$. Since the span of a quadric surface is $\PP^3$,
any quadric surface swept out by the lines parameterized by an
irreducible conic containing $p_1$ and $p_2$ must be contained in this
$\PP^3$. We conclude that the Zariski closure of the union of the
conics passing through $p_1$ and $p_2$ is contained in the Schubert
variety $\Sigma_{1,1}(F_{3} \subset F_4 = \overline{\Omega_1
  \Omega_2})$. By the previous paragraph, for every point $p$ in
$\Sigma_{1,1}(F_{3} \subset F_4 = \overline{\Omega_1 \Omega_2})$,
there exists a conic containing $p, p_1$ and $p_2$. Hence, the Zariski
closure of the union of conics in the fibers of $f$ containing $p_1$
and $p_2$ is precisely the Schubert variety $\Sigma_{1,1}(F_{3}
\subset F_4 = \overline{\Omega_1 \Omega_2})$. This Schubert variety is
isomorphic to $Gr(2,4)$ and is embedded in $\PP^9$ by the Pl\"{u}cker
embedding as a quadric fourfold \cite{Coskun2011}.  Its image under
$f$ is $\PP^3$ obtained by projecting $\Sigma_{1,1}(F_{3} \subset F_4
= \overline{\Omega_1 \Omega_2})$ from the line joining $p_1$ and
$p_2$. We conclude that over a $\PP^3$ in $\PP^5$, the fiber of $f$
contains a (possibly reducible) conic passing through $p_1$ and
$p_2$. In particular, fibers of $f$ are irreducible in codimension
two.  \smallskip

\noindent {\bf (2) $Z$ is the complete intersection of two quadric
  hypersurfaces in $\PP^{n+2}$.} Let $n \geq 3$. If $Z$ is a transverse
intersection of two smooth quadric hypersurfaces $Q_1$ and $Q_2$ in
$\PP^{n+2}$, then, by adjunction, $-K_Z = (n-1)H$. The linear system
$V \subset |H|$ is the set of hyperplanes that contain a general plane
$\Lambda$ in $\PP^{n+2}$ and $X$ is the blowup of $Z$ at the four
points $\Gamma = \{ p_1, \dots, p_4\} = Z \cap \Lambda$. The fibers of
$f: X \rightarrow \PP^{n-1}$ are quartic elliptic space curves. Hence,
reducible fibers consist of either the union of a line and a cubic or
two conics, where the cubic or either of the conics may be further
reducible.  \smallskip

In a smooth quadric hypersurface $Q$, the lines that pass through a
point $p \in Q$ sweep out the codimension one quadric cone $T_p Q \cap
Q$. Hence, the lines that pass through a point $p_i$ on $Z$ are
contained in $Y= T_{p_i}Q_1 \cap T_{p_i} Q_2 \cap Z$. Conversely,
since $Y$ is a cone with vertex at $p_i$, every point of $Y$ is
contained in a line passing through $p_i$. Since $Z$ is a transverse
intersection of the two quadrics $Q_1$ and $Q_2$, we conclude that $Y$
is a codimension two subvariety of $Z$.  \smallskip

Similarly, conics that pass through two points $p_1, p_2$ on $Z$ sweep
out a subvariety of $Z$ of codimension at least two. This can be
verified by a simple dimension count. In a homogeneous variety of the
form $Y = G/P$, where $G$ is a linear algebraic group and $P$ is a
parabolic subgroup, the space of rational curves in the class $\beta$
is irreducible and has dimension $-K_Y \cdot \beta + \dim(Y) - 3$
\cite{kim}. A quadric hypersurface in $\PP^{n+2}$ is homogeneous and
contains a $3n$-dimensional family of conics. By Kleiman's
Transversality Theorem \cite{kleiman}, the locus of conics that are
contained in the intersection of two quadrics has dimension
$3n-5$. Applying Kleiman's Transversality Theorem once more, we
conclude that the space of conics containing two general fixed points
has dimension $n-3$. Therefore, these conics sweep out at most a
codimension two subvariety of $Z$. The same argument bounds the
dimension of reducible fibers in the remaining examples, but in each
case we will give a much more explicit description of the reducible
fibers.
\medskip

\noindent{\bf (3) $Z$ is a smooth cubic hypersurface in $\PP^{n+1}$.} Let $n
\geq 3$.  If $Z$ is a smooth cubic hypersurface in $\PP^{n+1}$, then
$-K_Z = (n-1) H$, where $H$ denotes the hyperplane class. The linear
system $V \subset |H|$ is the set of hyperplanes that contain a
general line $\Lambda$ in $\PP^{n+1}$ and $X$ is the blowup of $Z$ at
the three points $\Gamma = \{ p_1, p_2, p_3\} = Z \cap \Lambda$. The
fibers of $f$ are plane cubic curves. Hence, any reducible fiber is a
union of a line and a conic, where the conic may also be reducible.
\smallskip

An {\em Eckardt point} $p$ on a cubic hypersurface $Z$ is a point
where $T_p Z \cap Z$ is a cone with vertex at $p$. A general cubic
hypersurface does not contain any Eckardt points and a smooth cubic
hypersurface can contain at most finitely many Eckardt points
\cite[Corollary 2.2]{coskun:cubic}. Since $\Lambda$ is a general line
and $Z$ has only finitely many Eckardt points, we may assume that the
points $\{p_1, p_2, p_3\} = Z \cap \Lambda$ are not contained in $T_p
Z \cap Z$ for any Eckardt point $p \in Z$.  \smallskip

Let $q \in Z$ be a point which is not an Eckardt point. Then the space
of lines in $Z$ passing through $q$ is a $(2,3)$ complete intersection
in $\PP T_p Z$ \cite[Lemma 2.1]{coskun:cubic}. Hence, the space of
lines passing through $q$ is a variety of dimension $n-3$. We conclude
that lines that pass through $q$ sweep out a subvariety of $Z$ of
dimension $n-2$.  Since every reducible fiber of $f$ consists of a
line passing through one of the points $p_1, p_2$ or $p_3$ and a
residual conic, we conclude that the reducible fibers of $f$ occur in
codimension 2.

\medskip

\noindent {\bf (4) $Z$ is  $\PP^2 \times \PP^2$.} Let $H_i$ denote the pullback of
the hyperplane class via the projection $\pi_i : Z \rightarrow
\PP^2$. Then $-K_Z = 3(H_1 + H_2)$. The linear system $V \subset |H_1
+ H_2|$ is the set of hyperplanes containing a general $\Lambda \cong
\PP^4$ in the Segre embedding of $\PP^2 \times \PP^2$ in
$\PP^8$. Since $(H_1 + H_2)^4 = 6$, $X$ is the blowup of $Z$ at the
six points $\Gamma = \{ p_1, \dots, p_6 \} = Z \cap \Lambda$. The
fibers of $f: X \rightarrow \PP^3$ are elliptic normal sextic curves
in $\PP^5$ and have class $3H_1^2H_2 + 3H_1H_2^2$. Hence, a reducible
fiber of $f$ has an irreducible component of degree $1$, $2$ or
$3$. These curves can have various cohomology classes.  \smallskip

First, if a fiber of $f$ contains a line $l$, then the cohomology
class of $l$ is either $H_1^2 H_2$ or $H_1 H_2^2$. Curves in these
classes are lines contained in a fiber of $\pi_1$ or $\pi_2$. By Lemma
\ref{d-points}, $l$ contains one of the points $p_i$. Since both
$\pi_1$ and $\pi_2$ have unique fibers containing $p_i$, we conclude
that $l$ has to be contained in one of the two $\PP^2$ containing
$p_i$. Conversely, the span of any line passing through $p_i$ and
$\Gamma$ is a $\PP^5$ containing $\Lambda$ and gives a reducible fiber
of $f$.  We conclude that the lines in the fibers of $f$ sweep out the
planes that contain the points $p_i$. In particular, the locus of
reducible fibers containing a line has codimension two in $Z$.
\smallskip

If the fiber contains an irreducible conic $C$, then $C$ has class
$H_1^2 H_2 + H_1H_2^2$. Note that an irreducible curve of class
$2H_1^2 H_2$ or $2 H_1 H_2^2$ has to be contained in a fiber of
$\pi_1$ or $\pi_2$, respectively. If two of the points $p_i, p_j$ were
contained in the same fiber of one of the projections $\pi_1$ or
$\pi_2$, then the line joining the two points would be contained in
$Z$. This would contradict the fact that $Z \cap \Gamma$ is a finite
set of points.  Hence, an irreducible curve with class $ 2H_1^2 H_2$
or $2 H_1H_2^2$ cannot contain two of the points of $\Gamma$. By Lemma
\ref{d-points}, we conclude that $C$ has class $H_1^2 H_2 + H_1
H_2^2$. The projection of $C$ to either factor is a line $L_i$. Hence,
$C$ is contained in $L_1 \times L_2 \subset \PP^2 \times \PP^2$. These
lines are determined by the image of the projection of the two points
$p_1$ and $p_2$ contained in $C$. Conversely, any conic contained in
$L_1 \times L_2$ and containing $p_1$ and $p_2$ lies in a fiber of
$f$. We conclude that the conics in the fibers of $f$ containing $p_1$
and $p_2$ sweep out the quadric surface $L_1 \times L_2$, hence form a
codimension two subvariety of $Z$.  \smallskip

Finally, if the fiber decomposes into a union of two irreducible
curves $C_1 \cup C_2$ of degree three, then the curves must have
classes $2H_1^2 H_2 + H_1 H_2^2$ and $H_1^2 H_2 + 2 H_1 H_2^2$,
respectively.  As in the previous paragraph, an irreducible cubic
curve passing through three of the base points cannot have class
$3H_1^2 H_2$ or $3H_1 H_2^2$. We can choose five of the points $p_1,
\dots, p_5$ in $\Gamma$ without any constraints. After reindexing, we
may assume that $C_1$ contains three of these points $p_1, p_2,
p_3$. The projection of $C_1$ by $\pi_2$ is a line and and the
projection of $C_1$ by $\pi_2$ is a conic. Since the projection of the
three points $p_1, p_2, p_3$ by $\pi_2$ are not collinear, there
cannot be any reducible fibers consisting of the union of two cubics.
\smallskip

We conclude that  the reducible fibers of $f$ occur in codimension 2.
\medskip

\noindent {\bf (5) $Z$ is the flag variety $F(1,2;3)$.} Let $Z=
F(1,2;3)$ be the flag variety parameterizing flags $W_1 \subset W_2$
in $k^3$, where $W_i$ is a subspace of dimension $i$. Alternatively, we can think of $Z$ as the variety
parameterizing pointed lines in $\PP^2$. The Picard group of $Z$ is
generated by the two Schubert divisors. Let $H_1$ denote the class of
the Schubert divisor parameterizing pointed lines $(p \in L)$ such
that $L$ contains a fixed point $q$ on $\PP^2$. Let $H_2$ denote the
class of the Schubert divisor parameterizing pointed lines $(p \in L)$
such that $p$ is contained in a fixed line $M$. Then $-K_Z = 2 (H_1 +
H_2)$. The linear system $V \subset |H_1 + H_2|$ is the set of
hyperplanes containing a general $\Lambda \cong \PP^4$ in the
Pl\"{u}cker embedding of $F(1,2;3)$ in $\PP^7$. Since $(H_1 + H_2)^3 =
6$, $X$ is the blowup of $Z$ at the six points $\Gamma = \{ p_1 ,
\dots, p_6 \} = Z \cap \Lambda$. The fibers of $f: X \rightarrow
\PP^2$ are elliptic normal sextic curves in $\PP^5$ and have class
$H_1^2 + 2H_1H_2 + H_2^2= 3H_1^2 + 3H_2^2$. A reducible fiber of $f$
has to contain an irreducible curve of degree one, two or three.
\smallskip

The flag variety $F(1,2;3)$ admits two projections $\pi_1: F(1,2;3)
\rightarrow Gr(1,3) \cong \PP^2$ and $\pi_2: F(1,2;3) \rightarrow
Gr(2,3) \cong (\PP^2)^*$. The reducible fibers can be easily described
by considering their projections via these two maps.  \smallskip

A line in $Z$ has cohomology class $H_1^2$ or $H_2^2$. Geometrically,
these classes parameterize pointed lines in $\PP^2$ where the line is
a fixed line $L$ or the point is a fixed point $q$,
respectively. Hence, there is a unique line of each kind passing
through a point $p_i \in Z$ parameterizing the pointed line $(q \in
L)$.  \smallskip

Conics in $Z$ can have cohomology class $H_1^2 + H_2^2$ or
$2H_i^2$. The projection of a conic with class $H_1^2 + H_2^2$  to both $\PP^2$ and $(\PP^2)^*$ is
a line. Therefore, there is a unique conic with class $H_1^2 + H_2^2$
containing two points $(q_1, L_1)$ and $(q_2, L_2)$.  It parameterizes
pointed lines $(q,L)$ such that $q \in \overline{q_1, q_2}$ and $L_1
\cap L_2 \in L$. There cannot be any conics in the class $2H_i^2$ that
pass through two points $p_1, p_2$ whose projections by $\pi_i$ are
distinct.  \smallskip

Finally, a cubic in $Z$ can have class $3H_i^2$ or $2H_i^2 +
H_j^2$. Note that there cannot be cubics of this type passing through
three general points. In the first case, the three points by the
projection $\pi_i$ has to coincide. In the second case, projection of
the curve under $\pi_i$ is a line. However, three general points are
not collinear under this projection. Hence, there are no reducible
fibers of this kind.  \smallskip

We conclude that $f$ has finitely many reducible fibers. 

\medskip

\noindent{\bf (6) $Z$ is $\PP^1 \times \PP^1 \times \PP^1$.} Let $H_i$
denote the pullback of the hyperplane class by the $i$-th projection
$\pi_i: Z \rightarrow \PP^1$. Then $-K_Z = 2(H_1 + H_2 + H_3)$. The
linear system $V \subset |H_1 + H_2 + H_3|$ is the set of hyperplanes
that contain a general $\Lambda \cong \PP^4$ in the Segre embedding of
$Z$ in $\PP^7$. Since $(H_1 + H_2 + H_3)^3 = 6$, $X$ is the blowup of
$Z$ at the six points $\Gamma = \{ p_1 , \dots, p_6 \} = Z \cap
\Lambda$. The fibers of $f: X \rightarrow \PP^2$ are elliptic normal
sextic curves in $\PP^5$ and have cohomology class $2(H_1H_2 + H_1H_3
+ H_2 H_3)$. A reducible fiber of $f$ must have an irreducible curve
of degree one, two or three.  \smallskip

The degree one component of a reducible fiber must have class $H_i H_j$. There is a
unique line in $Z$ in this class through every point. Curves of degree
two may have cohomology class $H_i H_j + H_i H_k$ or $2 H_i
H_j$. Curves in the latter class are necessarily reducible, so we can
concentrate on curves of degree two in the class $H_i H_j + H_i
H_k$. Since these curves have to pass through two points and these two
points do not have the same $i$-coordinate, we conclude that there
cannot be reducible fibers containing a curve of degree two. Finally, we consider
irreducible cubics in the fiber. These can have cohomology class $2H_i
H_j + H_i H_k$ or $H_1 H_2 + H_1 H_3 + H_2 H_3$. There cannot be
irreducible curves with the first class and there is a unique curve
with the second class passing through three points. We conclude that
$f$ has finitely many reducible fibers.

\smallskip Finally, we deal with the following slightly different case:
\smallskip

\noindent{\bf (7) $Z \arrow \PP^n$ is a double cover of $\PP^n$
  branched along a smooth quartic $Q$.} This case is
slightly different, since the line bundle $H=-\frac{1}{n-1} K_Z$
is ample but not very ample: $H$ is the pullback of $\OO_{P^n}(1)$
under the double cover map . Nevertheless, we can treat this case
in a very similar way to those above.
\smallskip

Let us denote the double cover by $d: Z \arrow \PP^n$. The linear
system $V \subset |H|$ is the pullback to $Z$ of the linear system of
hyperplanes through a given general point $p \in \PP^n$. On $Z$ this linear
system has base locus $\{p_1,p_2\} = d^{-1}{p}$, and $X$ is the blowup
of $Z$ at these two points. The fibers of $f$ are then preimages
$d^{-1}(L)$ of lines $L$ in $\PP^n$ through $p$: the preimage of a
line in $\PP^n$ is a double cover of $\PP^1$ branched over 4 points,
hence an elliptic curve.
\smallskip

A fiber $d^{-1}(L)$ of $f$ is reducible if and only if $L$ is a
bitangent line of the quartic $Q$ passing through $p$. Let us show
that the locus of such lines has codimension 2 in $\PP^{n-1}$. 
\smallskip

To see that it has codimension at least 2, it suffices to exhibit a curve in
$\PP^{n-1}$ (the space of lines in $\PP^n$ through $p$) which is
disjoint from the locus of bitangent lines. To do this, take a general
plane section of the quartic, which is a smooth quartic plane
curve. Such a curve has 28 bitangent lines (and these are exactly the
bitangents of $Q$ lying in the chosen plane).  A general point $p$ in the
plane does not lie on any of these bitangent lines, so the pencil of
lines in the plane through $p$ does not intersect the locus of
bitangents through $p$. 
\smallskip

To see that the locus of bitangents through $p$ has codimension equal
to 2, it suffices to show that this locus intersects any plane $\PP^2
\subset \PP^{n-1}$ in the space of lines through $p$. Such a plane
corresponds to a subspace $\PP^3 \subset \PP^n$, and this intersects $Q$ in a
quartic surface in $\PP^3$. The locus of bitangents to a quartic
surface covers all of $\PP^3$, so in particular there is such a
bitangent through any point. This completes the proof that the locus
of reducible fibers has codimension 2.
\smallskip

Finally, we note that any reducible fiber must be of the form $C_1 \cup
C_2$, where $C_1$ and $C_2$ are curves with $H \cdot C_i = 1$ and $p_i
\in C_i$.

\medskip

We summarize the preceding discussion in the following statement:

\begin{proposition} \label{prop-redfibers}
Let $Z$ be a Fano manifold of index $n-1$, and $V \subset |H|$ a
general linear system of dimension $n-1$. Let $f: X \arrow \PP^{n-1}$
be the elliptic fibration corresponding to $V$. Then $f$ has reducible
fibers in codimension 2.  
\end{proposition}
\section{The nef cone}\label{section-nef}
In this section, we calculate the nef cone of the variety
$X$, and thereby show that the first statement of Conjecture
\ref{conj-coneconjecture} is true for $X$. 
\smallskip

We preserve the notation from the previous section. Recall that $g: X
\rightarrow Z$ is the blowup of the Fano manifold $Z$ in $\Gamma = \{
p_1, \dots, p_r \}$ points.  Let $E_i$ denote the exceptional divisor
over $p_i$. Let $[r]$ denote the set $\{1, 2, \dots, r \}$. We will
abuse notation and denote the pullbacks $g^* H$ of divisor classes
simply by $H$.

\begin{theorem}\label{nef}
The nef cone of $X$ has the following description.
\begin{enumerate}
\item If $Z$ is a linear section of $Gr(2,5)$, then the nef cone of
  $X$ is the cone spanned by the divisor classes $\{ H - \sum_{i \in
    I} E_i \ | \ I \subseteq [5] \},$ where $H$ denotes the hyperplane
  class of $Gr(2,5)$.  \smallskip

\item If $Z$ is the smooth complete intersection of two quadrics, then
  the nef cone of $X$ is the cone spanned by $\{H - \sum_{i \in I} E_i
  \ | \ I \subseteq [4] \} ,$ where $H$ denotes the hyperplane class
  of $Z$.  \smallskip

\item If $Z$ is a smooth cubic hypersurface, then the nef cone of $X$
  is the cone spanned by $\{ H - \sum_{i \in I} E_i \ | \ I \subseteq
  [3] \},$ where $H$ denotes the hyperplane class of $Z$.  \smallskip

\item If $Z = \PP^2 \times \PP^2$, then the nef cone of $X$ is the
  cone spanned by the divisor classes $H_1, \ H_2, \ \mbox{and} \ \{
  H_1 + H_2 - \sum_{i \in I} E_i \ | \ I \subseteq [6] \},$ where
  $H_i$ denotes the pullback of $\OO_{\PP^2}(1)$ by the projection
  $\pi_i: Z \rightarrow \PP^2$.  \smallskip

\item If $Z = F(1,2; 3)$, then the nef cone of $X$ is the cone spanned
  by the divisor classes $H_1, \ H_2, \ \mbox{and} \ \{ H_1 + H_2 -
  \sum_{i \in I} E_i \ | \ I \subseteq [6] \},$ where $H_i$ denotes the
  two Schubert divisors.  \smallskip

\item If $Z = \PP^1 \times \PP^1 \times \PP^1$, then the nef cone of
  $X$ is the cone spanned by the divisor classes $H_1, \ H_2, H_3, \
  \mbox{and} \ \{ H_1 + H_2 + H_3 - \sum_{i \in I} E_i \ | \ I
  \subseteq [6] \},$ where $H_i$ denotes the pullback of
  $\OO_{\PP^1}(1)$ by the projection $\pi_i: Z \rightarrow \PP^1$.

\item If $Z$ is a double cover of $\PP^n$ branched along a smooth
  quartic, then the nef cone of $X$ is the cone spanned by the divisor
  classes $H$, $H-E_1$, $H-E_2$, $H-E_1-E_2$.
\end{enumerate}
\end{theorem}

\begin{proof}
  The exceptional divisors $E_i$ are isomorphic to $\PP^{n-1}$. Let
  $e_i$ denote the class of a line in $E_i$.  Let $l$ denote the class
  of a line in $Z$.  In cases (1), (2) and (3), there is a family of
  lines in $Z$ that contain exactly one of the points $p_i$. The
  proper transform of such a line in $X$ is an effective curve with
  class $l-e_i$.  The classes $$\{ l, e_1, \dots, e_r, l-e_1, \dots,
  l-e_r \}$$ span a subcone $C$ of the cone of curves $\Curv{X}$. The
  dual of this cone $\check{C}$ in $N^1(X)$ contains the nef cone. The
  cone $\check{C}$ also contains the divisors listed in the theorem.
  \smallskip

  Let $D$ be a divisor contained in $\check{C}$. We may express $D= aH
  - \sum_{i \in [r]} b_i E_i$. Since $D \cdot l = a$ and $D \cdot e_i
  = b_i$, we conclude that $a, b_i \geq 0$. After reordering the
  indices, we may assume that $b_1 \geq b_2 \geq \cdots \geq
  b_r$. Similarly, pairing $D$ with $j l - e_1 - \cdots - e_j$, we
  conclude that for any divisor $D \in \tilde{C}$, $j a \geq
  \sum_{i=1}^j b_i$. Hence, we can express $D$ as $$D = (a -
  \sum_{i=1}^r b_i) H + \sum_{j=1}^{r-1} (b_j - b_{j+1})(H -
  \sum_{i=1}^j E_i) + b_r ( H - \sum_{i=1}^r E_i).$$ We conclude that
  the dual cone $\check{C}$ is generated by the divisors listed in the
  theorem.  \smallskip

  Similarly, in cases (4) and (5), there are families of lines with
  class $l_1$ and $l_2$ dual to $H_1$ and $H_2$, respectively, in $Z$
  passing through exactly one of the points $p_i$. The proper
  transform of these lines in $X$ are effective curves with class
  $l_i- e_j $. The classes $$\{ l_1, l_2, e_1, \dots, e_r, l_1-e_1,
  \dots, l_1-e_r, l_2 - e_1, \dots, l_2 - e_r \}$$ span a subcone $C$
  of the cone of curves $\Curv{X}$. The dual of this cone $\check{C}$ in
  $N^1(X)$ contains the nef cone and the cone generated by the
  divisors listed in the theorem. In fact, by an identical argument,
  the divisors listed in the theorem generate the dual cone
  $\check{C}$. Given a divisor $D= a_1 H_1 + a_2 H_2 - \sum_{i=1}^r
  b_i E_i$ with $b_1 \geq \cdots \geq b_r$ in the dual cone
  $\check{C}$, by pairing with curves with class $j l_i - e_1- \cdots
  - e_j$, we conclude that $j a_i \geq \sum_{i=1}^j b_i$. Hence, we
  can express $D$ as
\begin{eqnarray*}
  D&=& (a_1 - b_1) H_1 + (a_2 - b_1) H_2 + \sum_{j=1}^{r-1} (b_j - b_{j+1})(H_ 1+ H_2 - \sum_{i=1}^j E_i) \\ &+& b_r ( H_1 + H_2 - \sum_{i=1}^r E_i).
\end{eqnarray*}

In case (6), there are families of lines with classes $l_1,
l_2$ and $l_3$ dual to $H_1, H_2$ and $H_3$, respectively, in $Z$
passing through exactly one of the points $p_i$. The proper transform
of these lines in $X$ are effective curves with classes $l_1 - e_i,
l_2 -e_i,$ and $l_3-e_i$. We thus obtain a subcone $C$ of the cone of
curves $\Curv{X}$ generated by curve classes $$\{l_1, l_2, l_3, e_1,
\dots, e_6, l_1 - e_1, \cdots, l_1-e_6, l_2-e_1, \dots, l_2-e_6, l_3 -
e_1, \dots, l_3 - e_6 \}.$$ As in the previous two paragraphs, the
divisors listed in the theorem generate the dual cone $\check{C}$.

Finally in case (7), the components of reducible fibers give curves on
$X$ with classes $l-e_1$, $l-e_2$, where the curve class $l$ satisfies
$H \cdot l =1$. So we obtain a subcone $C$ of the cone of curves
generated by $\{l-e_1,l-e_2,e_1,e_2\}$. Again, the divisors listed
generate the dual cone $\check{C}$. 

 \smallskip

To conclude the proof it remains to show that the divisors listed in
the theorem are nef. In cases (1)-(3) and (7), the divisor $H$, in cases
(4)-(5), the divisors $H_1$ and $H_2$ and in case (6), the divisors
$H_1, H_2, H_3$ are semi-ample, in particular nef, being the pullbacks of ample
divisors by morphisms. The divisor $$H - \sum_{i=1}^r E_i = -
\frac{1}{n-1} K_X$$ is base-point-free since this is the
divisor that defines the elliptic fibration $f: X \rightarrow
\PP^{n-1}$. All the other divisors listed in parts (1)-(6) of the
theorem are of the form $D= - \frac{1}{n-1} K_X + E_{i_1} + \cdots +
E_{i_j}$. If $C$ is an irreducible curve such that $C \cdot D < 0$,
then $C$ is necessarily contained in one of the exceptional divisors
$E_i$. However, since the exceptional divisor $E_i$ is a projective
space, any effective curve class on $E_i$ is proportional to
$e_i$. Since $e_i \cdot D \geq 0$, we obtain a contradiction. It
follows that the cone $\check{C}$ is the nef cone. Furthermore, by the
Base-Point-Free Theorem \cite[Theorem 3.3]{KollarMori1998}, all of
these divisors are semi-ample. Both $D$ and $D-\frac{1}{n-1}K_X$ are
nef, and moreover $D-\frac{1}{n-1}K_X=- \frac{2}{n-1}K_X +E_{i_1} +
\cdots + E_{i_j}$ has top self-intersection number
\begin{align*}
2^n H^n - 2^n \sum_{i \not\in \{i_1, \dots, i_j \} } E_i^n  -  \sum_{i \in \{i_1, \dots, i_j \} } E_i^n  = 2^n r - 2^n (r-j) - j = (2^n -1)j > 0
\end{align*}
so is big. Hence, by the Base-Point-Free Theorem, $D$ is semi-ample. 
\end{proof}

\begin{corollary}
The first statement of Conjecture \ref{conj-coneconjecture} holds for $X$.
\end{corollary} 

\begin{proof} By Theorem \ref{nef},
$\Nef{X}^e$ equals $\Nef{X}$, a rational polyhedral cone. The
automorphism group $\Aut(X)$ acts on $N^1(X)$ as a subgroup of
$\operatorname{GL}(N^1(X)_\ZZ)$ and moreover preserves $\Nef{X}$. This implies that
the image of $\Aut(X)$ in $\operatorname{GL}(N^1(X)_\ZZ)$ must be finite, since any
automorphism $g$ must permute the primitive integral vectors on the
(finitely many) extremal rays of $\Nef{X}$, and this permutation
determines the image of $g$ in $\operatorname{GL}(N^1(X)_\ZZ)$.

Therefore, we have a finite group of integral linear transformations
acting on the rational polyhedral cone $\Nef{X}$. It is then
straightforward to produce a rational polyhedral fundamental domain
$\Pi$ for the action. 
\end{proof}

\begin{remark}
  In fact, one can show that in all but two of our examples the
  automorphism group $\Aut(X)$ is trivial, so that
  $\Nef{X}^e$ is the fundamental domain. The exceptions are
  the following
\begin{itemize}
\item The case $Z=Gr(2,5)$: here there are automorphisms inducing any
  permutation of the 5 base-points of the linear system $V$, so we get
  a group of automorphisms isomorphic to the symmetric group $S_5$.
\item The case $Z \arrow \PP^n$ is a double cover branched over a
  quartic: here we have the involution switching the two sheets of the
  covering. This restricts to the hyperelliptic involution on any
  smooth fiber of $X \arrow \PP^{n-1}$.
\end{itemize}
\end{remark}

\section{Flops of $X$} \label{section-flops}
In this section we record some facts we need about flops. We show that
the flops we need for the proof of Conjecture
\ref{conj-coneconjecture} do in fact exist, and prove a lemma about
classes of certain effective curves on flops. Finally, we show that in
our examples, flops preserve the property of having rational
polyhedral nef cone.

\begin{lemma} \label{lemma-flopclass} Let $Y$ be a projective
  variety, and let  $g: Y \dashrightarrow Y'$ be
  the flop of an extremal ray $R \subset \Curv{Y} \cap K^\perp$. Let
  $\Gamma$ be a curve which intersects exactly one curve $C$ with $[C]
  \in R$, and suppose the intersection is transverse and consists of
  $k$ points. Then the proper transform $\Gamma'$ of $\Gamma$ on $Y'$
  has numerical class $[\Gamma] + k [C]$.
\end{lemma} \begin{proof} Let $U \subset Y$ and $U' \subset Y'$ be the
  two maximal open subsets that are isomorphic under the flop $g$.
  Choose a very ample divisor $A$ which intersects $C$ and $\Gamma$
  transversely and is disjoint from $C \cap \Gamma$. Let $A'$ denote
  the strict transform of $A$ in $Y'$.  Since $g$ is an isomorphism on
  $U$, we get $A \cdot \Gamma$ transverse intersection points of $A'$
  and $\Gamma'$ on $U'$.  Let $C'$ be the cocenter of the flop. Next,
  we must consider intersection points of $A'$ and $\Gamma'$ along
  $C'$. The divisor $A'$ has multiplicity $A \cdot C$ at every point
  of $C'$, and $\Gamma'$ meets $C'$ transversely in $k$ points. So for
  a general choice of $A$, we obtain a further contribution of $k(A
  \cdot C)$ to the intersection number $A' \cdot \Gamma'$.  We
  conclude that for a general very ample divisor $A$, we have
\begin{align*}
A' \cdot \Gamma' &= A \cdot \Gamma + k ( A \cdot C).
\end{align*}
Since this is true for all general very ample divisors $A$, and these
span $N^1(Y)$, we conclude that
\begin{align*}
[\Gamma'] = [\Gamma] + k[C].
\end{align*}
\end{proof}

\smallskip We will use this lemma to calculate the numerical classes
of fiber components on flops of our varieties. These classes will be
important in later sections, since they give bounds on the nef cones
of these flops. We will see in the next section that in each case it
suffices to consider a short list of flops. The numerical classes of
fiber components on these flops are recorded below. In each case, $F$
denotes the class of a fiber of the elliptic fibration $f: X \arrow
\PP^{n-1}$.

\begin{enumerate}
\item {\bf $Z=Gr(2,5)$}. In this case, we need to flop the locus of lines,
  conics, and cubics through one of the base-points $p_1$. This creates
  numerical classes of fiber components of the following kinds:
\begin{align*}
 F+(l-e_1), \quad & F+(2l-e_1-e_i), \quad  F + (3l-e_1-e_i-e_j).
\end{align*}

\item {\bf $Z$ is an intersection of quadric hypersurfaces in
    $\PP^{n+2}$.} In this case, we need to flop the locus of lines and conics
  through a given point $p_1$; it is not necessary to flop cubics. As
  in the previous case, this creates classes of the following kinds:
\begin{align*}
 F+(l-e_1), \quad & F+(2l-e_1-e_i).
\end{align*}
\item {\bf $Z$ is a cubic hypersurface in $\PP^{n+1}$}. In this case,
  we will not need to perform any flops on $X$. \smallskip

\item {\bf $Z=\PP^2 \times \PP^2$.} Here it will be necessary to perform
  longer sequences of flops: it may be that we need to flop a
  component of a fiber, then on the new space flop the proper
  transform of the other component, and so on. One can apply Lemma
  \ref{lemma-flopclass} repeatedly to show that the result is always a
  class of the form $mF+C$, where $C$ is the original class we
  flopped, and $m$ is a positive integer. In our cases the relevant
  classes on these flops are 
\begin{align*}
mF &+ (l_1-e_1) \, (m=1,2) \\
mF &+ (l_1+l_2-e_i-e_j) \, (m=1,2,3) \\
mF &+ (2l_1+l_2-e_1-e_j-e_k) \, (m=1,\ldots,4)
\end{align*}

\item {\bf $Z=F(1,2;3)$, $Z=\PP^1 \times \PP^1 \times \PP^1$.} In these
  cases, we will avoid explicit calculations by using facts
  specific to threefolds.
\end{enumerate}

\smallskip

\begin{lemma} \label{lemma-flops} Let $X$ be one of our examples. Let
  $R$ be an isolated extremal ray of $\Curv{X}$ which lies in $K_X^\perp$. Then
  the flop of $R$ exists. More generally, if $X'$ is an SQM of $X$
  obtained by a sequence of flops of classes in $K^\perp$, and $R$ is
  an isolated extremal ray of $\Curv{X'}$ which lies in $K_{X'}^\perp$, then
  the flop of $R$ exists.
\end{lemma} \begin{proof} For any $X'$ as in the statement, the
  anti-canonical divisor $-K_{X'}$ is semi-ample, so there exists a
  $\QQ$-divisor $\Delta$ such that $(X',\Delta)$ is a klt Calabi--Yau
  pair. Given a ray $R$ as in the statement of the lemma, it suffices
  to find an effective divisor $D$ such that $D \cdot R <0$. For then,
  $(X,\Delta + \epsilon D)$ is a klt pair for sufficiently small
  $\epsilon$, and $R$ is a $(K_X+\Delta+\epsilon D)$-negative extremal
  ray. By the Cone Theorem, the contraction of $R$ exists. Therefore,
  by Birkar--Cascini--Hacon--McKernan \cite[Corollary 1.4.1]{BCHM},
  the flop of $R$ exists.  \smallskip

To show that the necessary effective divisor $D$ exists, we use the
theorem of Boucksom--Demailly--P\u{a}un--Peternell \cite[Theorem
0.2]{BDPP}, which asserts that the cone of pseudoeffective divisors
$\overline{\operatorname{Eff}(X')}$ is dual to the cone of moving
curves on $X'$. Now suppose $R$ is an extremal ray in $K^\perp$
spanned by a curve $C$. Any curve $C'$ numerically a multiple of $C$
must be contained in a fiber of the elliptic fibration $f: X
\rightarrow \PP^{n-1}$, and moreover cannot be a whole fiber (since
the class $F$ does not span an extremal ray). Since $X$ has reducible
fibers in codimension 2, the same is true of any flop $X'$. Hence, in
particular, the curve $C$ cannot be moving. Since the ray $R$ spanned
by $C$ is an isolated extremal ray, the closure of the cone of moving
curves does not contain $R$. By the theorem of
Boucksom--Demailly--P\u{a}un--Peternell, there is an effective divisor $D$
with $D \cdot C <0$, as required. \end{proof}

\section{The movable cone} \label{section-movablecone} In this
section, we prove part (2) of Conjecture \ref{conj-coneconjecture} for
our examples.  We first study the action of the pseudo-automorphisms
$\PsAut(X/ \PP^{n-1})$ preserving the elliptic fibration $f: X
\rightarrow \PP^{n-1}$ on the relative effective movable cone
$\Mov{X/\PP^{n-1}}^e$. We find a rational polyhedral fundamental domain
for this action.  We further show that the union of the nef cones of
finitely many SQMs lie over this fundamental domain. Since these nef
cones are rational polyhedral cones, we obtain a rational polyhedral
fundamental domain for the action of pseudo-automorphisms on
$\Mov{X}^e$.  We use a combination of general results from minimal
model theory and explicit analysis of our examples. \smallskip

As explained in the introduction, the effective nef cones of all the
SQMs of a variety $X$ are contained in $\Mov{X}^e$. The next lemma
shows that in our examples the converse is also true: the effective
movable cone is the union of the effective nef cones of the SQMs of
$X$. Given a contraction morphism $f: X \arrow Y$, an effective
divisor $D$ on $X$ is called {\it $f$-vertical} if $f_*(D) \neq Y$.

\begin{lemma}\label{logflip}
  Let $X$ be a smooth projective variety such  that some power
  of $-K_X$ defines a contraction morphism $f: X \arrow Y$ with
  $\operatorname{dim } Y = \operatorname{dim }X -1$. Suppose
  that every $f$-vertical divisor on $X$ is a multiple of $-K_X$. Then
\begin{align*}
\Mov{X}^e &= \bigcup \Nef{X'}^e,
\end{align*}
where the union on the right is over all SQMs $X' \dashrightarrow X$.
\end{lemma}
\begin{proof} We must show that if $D$ is any effective $\RR$-divisor
  on $X$ whose numerical class lies in $\Mov{X}$, then that class
  belongs to the nef cone of some SQM of $X$. Since, by assumption,
  $-K_X$ is already semi-ample on $X$, we may assume that $D$ is not a
  multiple of $-K_X$. Since there are no $f$-vertical divisors other
  than multiplies of $-K_X$, we have $D \cdot F >0$, where $F$ is the
  generic fiber of $f: X \arrow Y$. Therefore, $D$ is $f$-big.  Hence,
  by \cite[Lemma 3.23]{KollarMori1998}, the $\RR$-divisor $-mK_X+D$ is
  big for sufficiently large values of $m$. Dividing by $m$, we see
  that $-K_X+\epsilon D$ is big for sufficiently small positive values
  of $\epsilon$.

Now choose a positive integer $r$ such that $-rK_X$ is base-point-free,
and let $\Delta$ be a smooth divisor in the linear system
$|-rK_X|$. Setting $\Theta = \frac{1}{r} \Delta$, the pair
$(X,\Theta)$ is klt, hence so is $(X, \Theta + \epsilon D)$ for
sufficiently small positive numbers $\epsilon$. Moreover, 
\begin{align*}
\Theta + \epsilon D &\equiv -K_X + \epsilon D
\end{align*}
is big for sufficiently small $\epsilon$ by the previous paragraph,
and
\begin{align}\label{DDD}
K_X + \Theta + \epsilon D &\equiv \epsilon D
\end{align}
is effective. By Birkar--Cascini--Hacon--McKernan \cite{BCHM}, the
pair $(X, \Theta+\epsilon D)$ has a minimal model $X'$. By Equation
(\ref{DDD}), the proper transform of $D$ on $X'$ is nef. Finally,
since $D \in \Mov{X}$, there cannot be a $D$-negative divisorial
contraction, so the birational map $X \dashrightarrow X'$ is an
SQM.  \end{proof}

\begin{corollary}
The conclusion of Lemma \ref{logflip} holds for each of our examples
$X$. \end{corollary}
\begin{proof}
  By construction each of our examples has an elliptic fibration $f: X
  \arrow \PP^{n-1}$. We need to show that there are no $f$-vertical
  divisors other than multiples of $-K_X$. Each fibration $f: X \arrow
  \PP^{n-1}$ is equidimensional, so any irreducible vertical divisor
  $D$ must be a component of a subset $f^{-1}(Y)$ where $Y \subset
  \PP^{n-1}$ has codimension 1. By Proposition \ref{prop-redfibers},
  the fiber of $f$ over any codimension-one point is irreducible, so we
  must have $D=f^{-1}(Y)$, in other words some multiple of $D$ is
  pulled back from a divisor on $\PP^{n-1}$. The image of
  $\Pic(\PP^{n-1}) \hookrightarrow \Pic(X)$ is generated by some
  (rational) multiple of $-K_X$, so $D = -qK_X$ for some rational
  number $q$, completing the proof.
\end{proof}

The next lemma exploits the combinatorial structure of the chamber
decomposition of the movable cone to reduce the problem of
finding a fundamental domain to a more ``local'' one. The upshot is
that we do not need to understand all SQMs of a given variety $X$, but
only those which are ``close'' to $X$.
\begin{lemma} \label{lemma-conestrick}
Let $X$ be a variety such that 
\begin{align*}
\Mov{X}^e &= \bigcup \Nef{X'}^e,
\end{align*}
where the union on the right is over all SQMs $X' \dashrightarrow
X$. Suppose that there is a collection of SQMs $\{X_i\}_{i \in I}$ with the
following property: for each SQM $X_\alpha$ of $X$ such that
$\Nef{X_\alpha}$ shares a codimension-one face with one of the cones
$\Nef{X_i}$, there exists an SQM pseudo-automorphism $\varphi$ of $X$
such that $\varphi_*(\Nef{X_\alpha}) \subset U$, where $U$ denotes the
union $\bigcup_{i \in I} \Nef{X_i}$. Then $\PsAut^*(X) \cdot U =
\Mov{X}^e$.
\end{lemma}
\begin{proof} Form a graph with a vertex for the nef cone of each SQM
  and an edge connecting $\Nef{X_\alpha}$ and $\Nef{X_\beta}$ if there
  is a log flip $X_\alpha \dashrightarrow X_\beta$. By Lemma
  \ref{logflip}, each SQM $X_\alpha \dashrightarrow X$ consists of a
  finite sequence of log flips. Therefore, any two vertices of the
  graph are at a finite distance in the graph metric.  \smallskip

We will prove the lemma by induction on the distance $d$ of a given
nef cone $\Nef{X_\alpha}$ from $U$. For $d=1$, the assumption of the
lemma gives $\Nef{X_\alpha} \subset \PsAut^*(X) \cdot U$. Now suppose
that all nef cones at a distance no more than $d$ are contained in
$\PsAut(X) \cdot U$, and suppose $\Nef{X_\alpha}$ is a cone at
distance $d+1$. Then there is a neighboring nef cone $\Nef{X_\beta}$
at distance $d$ from $U$. By the induction hypothesis, there is a
pseudo-automorphism $\varphi$ such that $\varphi_*(\Nef{X_\beta})
\subset U$. Now pseudo-automorphisms act on the graph of nef cones via
automorphisms, so $\varphi_*(\Nef{X_\alpha})$ must either be contained
in $U$ or be a neighbor of a nef cone in $U$. In the second case, by
assumption there is another pseudo-automorphism $\psi$ such that
$\psi_* \varphi_*(\Nef{X_\alpha}) \subset U$, so in either case we
have $\Nef{X_\alpha} \subset \PsAut^*(X) \cdot U$. By induction,
$\PsAut(X) \cdot U$ contains every nef cone $\Nef{X_\alpha}$,
therefore equals $\Mov{X}^e$. \end{proof}

The next lemma describes the relative effective movable cone of the
elliptic fibration $f: X \rightarrow \PP^{n-1}$. We recall that a
Cartier divisor $D$ is called {\em $f$-movable} (respectively, {\em
  $f$-effective}) if codim(Supp(Coker$(f^* f_* \OO_X(D) \rightarrow
\OO_X(D)))) \geq 2$ (respectively, $f_* \OO_X(D) \not= 0$).

\begin{proposition} \label{prop-movablecone}
The relative effective movable cone of $f: X \arrow \PP^{n-1}$ is 
\begin{align*}
\Mov{X/\PP^{n-1}}^e = B:= \left\{ x \in N^1(X/\PP^{n-1}) : x \cdot F >0 \right\}
\cup \{0\},
\end{align*}
where $F$ denotes the class of a fiber of $f$.
\end{proposition}
\begin{proof} The inclusion $\Mov{X/\PP^{n-1}}^e \subseteq B$ is easy
  to see. An $f$-effective divisor $D$ has nonnegative degree on the
  generic fiber of $f$, and has degree 0 only if it is vertical. Since
  the only vertical divisors are multiples of $-\frac{1}{n-1}K_X$, a
  vertical divisor is zero in $N^1(X/\PP^{n-1})$.

To see the reverse inclusion, we use a standard argument involving
Grauert's theorem on semicontinuity of cohomology \cite[Corollary
12.9]{Hartshorne1978}. Suppose $D$ is a divisor on $X$ such that $D
\cdot F >0$. Then the restriction $D_{|X_y}$ to any (fixed)
irreducible fiber $X_y$ is ample. Replacing $D$ by a positive multiple
if necessary, we have that $D_{|X_y}$ is base-point-free, in
particular, effective. Grauert's theorem implies that a section of
$D_{|X_y}$ is the restriction of a section $s \in O_X(D)\left(
  f^{-1}(U) \right)$ for some open set $U \subset \PP^{n-1}$. Hence,
$f_*(O_X(D)) \neq 0$; that is, $D$ is $f$-effective. To see that $D$
is $f$-movable, replacing $D$ by a positive multiple, we can assume
that the restriction $D_{|X_y}$ to every irreducible fiber $X_y$ is
base-point-free. Again by Grauert's theorem, a section of $D_{|X_y}$
is the restriction of a section $s \in O_X(D)\left( f^{-1}(U) \right)$
for some open set $U \subset \PP^{n-1}$. Since $D_{|X_y}$ is
base-point-free, this shows that the support of the sheaf
$\operatorname{Coker } \left( f^*f_* O_X(D) \arrow O_X(D) \right)$
does not contain any point of any irreducible fiber of $f$. We saw in
Section \ref{section-grassmannian} that $f$ has reducible fibers in
codimension 2, so $D$ is $f$-movable. We have shown that every integral
point of the cone $B$ belongs to $\Mov{X/\P}^e$. Since $B$ is spanned
by its integral points, this completes the proof of the reverse
inclusion. \end{proof}

Now we will study the action of pseudo-automorphisms on the movable
cone in more detail. The action on $N^1(X)$ is somewhat complicated,
but fortunately we do not need to understand it in detail. It will
suffice to understand the action on the space $N^1(X/\PP^{n-1})$. The
generic fiber $X_{\eta}$ of the fibration $f: X \rightarrow \PP^{n-1}$
is an elliptic curve. Given two rational points $p, q \in X_{\eta}$,
their difference $p-q \in \Pic^0(X_{\eta})$ defines a translation on
$X_{\eta}$, which extends to a pseudo-automorphism of $X$. The group
of translations is called the {\em Mordell-Weil group} and is denoted
by MW$(f)$. The hyperelliptic involution on $X_{\eta}$ also defines a
pseudo-automorphism of $X$. In fact, the group of relative
pseudo-automorphisms $\PsAut(X/\PP^{n-1})$ is the $\ZZ_2$-extension
of MW$(f)$ generated by the hyperelliptic involution. We will now
describe a fundamental domain for the action of $\PsAut(X/\PP^{n-1})$
on $N^1(X/\PP^{n-1})$.  \smallskip

 Observe that  
\begin{align*}
N^1(X/\PP^{n-1}) &\iso N^1(X) / \left< -K_X \right> \\
& = \left(\Pic(X) / \operatorname{ker } \rho   \right) \otimes \RR \\
&= \Pic(X_\eta) \otimes \RR,
\end{align*}
where $\rho$ denotes the natural restriction homomorphism $\Pic(X)
\arrow \Pic(X_\eta)$). An element
$x \in \mbox{MW}(f)$ acts on $N^1(X/\PP^{n-1})$ by
\begin{align*}
\varphi_x: y &\mapsto y + (y \cdot F) x.
\end{align*}
The involution $\iota$ acts by 
\begin{align*}
\iota: y \mapsto 2(y \cdot F)E_1 -y.
\end{align*}
Using this information, we obtain the following.

\begin{proposition} \label{prop-funddomain} Let $X$ be one of our
  examples. Then $N^1(X/\PP^{n-1})$ has a basis consisting of sections
  of the elliptic fibration $f: X \arrow \PP^{n-1}$. If $S_1,\ldots,
  S_k$ is a basis of rational sections $($where $k=\dim
  N^1(X/\PP^{n-1})$$)$ then a fundamental domain for the action of
  $\PsAut(X/\PP^{n-1})$ on $\Mov{X/\PP^{n-1}}^e$ is the following:
\begin{align*}
V= \RR_+ \cdot \left\{  S_1 + \sum_i a_i(S_i-S_1) : 0 \leq a_i \leq 1,
  \sum_i a_i \leq \frac{k-1}{2} \right\}.
\end{align*}
\end{proposition}
\begin{proof} First, by Theorem \ref{nef}, we see that the kernel of
  $N^1(X) \arrow N^1(X/\PP^{n-1})$ is spanned by $-\frac{1}{n-1}K_X$
  in each example. So $\dim N^1(X/\PP^{n-1}) = \dim N^1(X) -1$. Now
  $N^1(X)$ has a basis consisting of the exceptional divisors
  $\{E_i\}$ together with a basis of $N^1(Z)$, where $Z$ is the
  underlying Fano manifold. The exceptional divisors are always
  sections of the elliptic fibration, and their classes in
  $N^1(X/\PP^{n-1})$ remain linearly independent. So to complete the
  proof of the first claim, in each case, we must identify $\rho(Z)-1$
  additional rational sections of the fibration whose classes in
  $N^1(X/\PP^{n-1})$ are linearly independent of the classes $E_i$. If
  $\rho(Z)=1$, there is nothing to do. In the other cases, the
  rational sections we need are as follows:
\begin{enumerate}
\item $Z= \PP^2 \times \PP^2$: here an additional rational section is
  given by the preimage of the line in (say) the first copy of $\PP^2$
  through the points $\pi_1(p_1)$ and $\pi_1(p_2)$. This divisor has
  class $H_1 - E_1 -E_2$.
\item $Z = F(1,2;3)$: here an additional rational section is given by the
  divisor of all pointed lines which pass through the intersection
  point of the pointed lines corresponding to $p_1$ and $p_2$. This
  divisor has class $H_1-E_1-E_2$.
\item $Z=\PP^1 \times \PP^1 \times \PP^1$: here the two additional
  rational sections can be chosen to be the pullbacks of the
  projections $\pi_1(p_1)$ and $\pi_1(p_2)$. These have classes
  $H_1-E_1$ and $H_1-E_2$.
\end{enumerate}
This completes the proof of the first statement. To prove the second
statement, note first that by Proposition \ref{prop-movablecone} the cone
$\Mov{X/\PP^{n-1}}^e$ is the span of the affine hyperplane
\begin{align*}
W = \left\{ x \in N^1(X/\PP^{n-1}) \, | \, x \cdot F =1 \right\}.
\end{align*}
Furthermore, the action of $\PsAut(X/\PP^{n-1})$ preserves $W$. So it
suffices to prove that the region
\begin{align*}
V_1 := \left\{  S_1 + \sum_i a_i(S_i-S_1) : 0 \leq a_i \leq 1,
  \sum_i a_i \leq \frac{k-1}{2} \right\}
\end{align*}
is a fundamental domain for the action of $\PsAut(X/\PP^{n-1})$ on the
hyperplane $W$.
\smallskip

For definiteness, let us fix $S_1$ as the base-point of our group of
translations.  Then we have a basis of MW$(f)$ given by the
differences $\{S_i-S_1\}$. By the formula for the action of MW$(f)$,
these differences acts on $W$ by translations. We can write an
arbitrary element of $W$ in the form
\begin{align*}
S_1 + \sum_{i=2}^k a_i (S_i - S_1)
\end{align*}
for some real numbers $a_i$, and then applying appropriate elements of
MW$(f)$ we can transform this to a class with $0 \leq a_i \leq 1$ for
all $i$. Moreover, it is clear that two classes with $0 < a_i <1$ for
all $i$ are not in the same orbit of MW$(f)$. So the ``cubical''
region 
\begin{align*}
C= \left\{ S_1 + \sum_{i=2}^k a_i (S_i-S_1) \, | \, 0 \leq a_i \leq 1 \right\}
\end{align*}
is a fundamental domain for the action of MW$(f)$ on $W$. Finally, to
obtain a fundamental domain for the whole group $\PsAut(X/\PP^{n-1})$,
we must consider the action of the involution $\iota$. The formula for
the action of $\iota$ shows that $\iota$ acts on $W$ as follows:
\begin{align*}
S_1 + \sum_{i=2}^k a_i (S_i -S_1) & \mapsto (1+\sum_{i=2}^k a_i)S_1 -
\sum_{i=2}^k a_i S_i.
\end{align*}
If we then apply translation by the element $\sigma = \sum_i (S_i - S_1) \in
MW(f)$, we obtain $S_1 + \sum_i (1-a_i) S_i$. Therefore,
the composition $t_\sigma \circ \iota$ maps the region $C$ to itself
via the formula
\begin{align*}
  t_\sigma \circ \iota : S_1 + \sum_{i=2}^k a_i (S_i -S_1) & \mapsto
  S_1 + \sum_{i=2}^k(1- a_i) (S_i -S_1).
\end{align*}
Since all the coefficients of points in $C$ satisfy $0 \leq a_i \leq
1$, at least one of the quantities $\sum_{i=2}^k a_i$ and
$\sum_{i=2}^k(1- a_i)$ does not exceed $\frac{k-1}{2}$. This shows
that $V_1$ is indeed a fundamental domain of the action on the
hyperplane $W$. \end{proof}

\begin{proposition} \label{prop-nefcones} For each $X$ in our list and
  each SQM $X'$ of $X$ obtained by a sequence of flops of fiber
  components, the nef cone $\Nef{X'}$ is rational
  polyhedral. Moreover, it is spanned by semi-ample divisors; in particular,
  $\Nef{X'}^e=\Nef{X}$.
\end{proposition}
\begin{proof}
We will prove the dual statement that the closed cone of curves
  $\Curv{X'}$ is rational polyhedral. 
Let $R$ be an extremal ray of $\Curv{X'}$ which lies in $K^\perp$
  and is different from the ray spanned by $F$. Then there is
  an effective divisor $D'$ on $X'$ whose base locus is a union of
  fiber components and such that $D' \cdot R <0$. To see this, start
  with an ample divisor $A$ on $X$ such that $A \cdot F =1$. Applying
  pseudo-automorphisms  to $A$ produces effective
  divisors of the form
\begin{align}\label{DDDD}
D &= A + \sum_i a_i (S_i -S_1) -\frac{m}{n-1}K_X, 
\end{align}
where the $a_i$ are any chosen integers, $m$ is some positive integer,
and the $S_i$ are a basis of
sections of $f$. Now one can check that since $R$
is not the ray spanned by the class $F$, we have $(S_i-S_1) \cdot R
\neq 0$ for some $i$. Moreover $-K_X \cdot R =0$, and so we can choose
the integers $a_i$ suitably so that 
\begin{align*}
D \cdot R &= A \cdot R + \sum_i a_i (S_i-S_1) \cdot R < 0.
\end{align*}
Note that since $D$ is the proper transform of the ample divisor $A$
under a pseudo-automorphism $X \dashrightarrow X$ over $\PP^{n-1}$, the
base locus of $D$ consists only of fiber components. Finally, let $D'$
be the proper transform of $D$ on $X'$: since $X'$ is obtained from
$X$ by flopping fiber components, the base locus of $D'$ is still a
union of fiber components.

To prove the proposition, now suppose that $\Curv{X'}$ is not rational
polyhedral. By the Cone Theorem, there must be an infinite sequence
$\{C_i\}$ of irreducible curves on $X'$ such that the corresponding
rays $R_i$ are extremal rays of $\Curv{X'}$ and converge to a ray in
$K^\perp$. Since there are only finitely many classes of curves that
lie in the fibers of $f$, we may assume that the curves $C_i$ are not
fiber components.  Then $D \cdot C_i \geq 0$ for all divisors $D$ of
the form given by Equation (\ref{DDDD}), and hence $D \cdot R \geq 0$
for all such divisors $D$. Then, by the previous paragraph, $R$ must
be the ray spanned by the class $F$ of a fiber. By Lemma
\ref{lemma-flopclass}, the class $F$ is always in the interior of the
top-dimensional face $\Curv{X'} \cap K^\perp$, so it is impossible
that $\operatorname{lim}_i R_i = \RR_+ F$. This is a contradiction, so
$\Curv{X'}$ and hence $\Nef{X'}$ are rational polyhedral cones.

Finally, to prove that $\Nef{X'}$ is spanned by semi-ample divisors,
we apply the Base-point-free Theorem \cite[Theorem
3.3]{KollarMori1998}. If $X'$ is obtained from $X$ by a sequence of
flops, then $-K_{X'}$ is base-point-free, in particular nef. We claim that if $D$ is any nef Cartier divisor on $X'$, then $D$ is
semi-ample. If $D$ is a multiple of $-K_{X'}$ there is nothing to
prove, so assume it is not. The class $aD-K_{X'}$ is
nef for any $a \geq 0$, so by \cite[Proposition 2.61]{KollarMori1998}
$aD-K_{X'}$ is big if and only its top self-intersection number is
strictly positive. Any iterated self-intersection of $-K_{X'}$ is an
effective cycle, so by the Nakai--Moishezon criterion, for any $a>0$ we have
\begin{align*}
\left(aD-K \right)^n \geq a D \cdot (-K)^{n-1} = a' D \cdot F
\end{align*}
(where $a'=a \cdot (n-1)^{n-1}$). Since $D$ is nef, if it were the
case that $D \cdot F=0$, then we would have $D \cdot C =0$ for every
fiber component $C$ on $X'$. By Theorem \ref{nef} and Lemma
\ref{lemma-flopclass} fiber components on any flop span a
codimension-one subspace $V \subset N_1$, and the dual space $V^{\perp}
\subset N^1$ is spanned by $-K$. Since by assumption $D$ is not a
multiple of $-K$, we conclude that $D \cdot F>0$. This shows that for
any $a>0$ the class $aD-K_{X'}$ is nef and big, so by the
Base-point-free Theorem $D$ is semi-ample.
\end{proof}

\begin{corollary}
If $X$ is one of our examples and $X'$ is an SQM of $X$, then any
$K$-trivial extremal ray $R$ of $\Curv{X'}$ is isolated, and the flop
of $R$ exists.
\end{corollary}
\begin{proof}
  By the previous proposition, $\Curv{X'}$ is rational polyhedral, so
  any extremal ray $R$ is isolated. The flop of $R$ then exists by
  \ref{lemma-flops}.
\end{proof}

\begin{proposition} \label{prop-covering} For each $X$
  in our list, there is a finite collection $\left\{ X_i \right\}_{i
    \in I}$ of SQMs of $X$ such that the cone $V$ is contained in the
  image of  the union 
\begin{align*}
\bigcup_{i\in I} \Nef{X_i}^e
\end{align*}
under the projection map $N^1(X) \arrow N^1(X/\PP^n)$.
\end{proposition} \begin{proof} The first step of the proof is to
  prove that the region $V$ is covered by the union of a finite
  collection of relative effective nef cones $\Nef{X_i/\PP^{n-1}}^e$,
  where the $X_i$ are SQMs of $X$ obtained by flopping classes of
  fibral curves. (Note that these cones are automatically rational
  polyhedral, since there are only finitely many classes of fiber
  components on any SQM, and the relative effective nef cone is the
  same as the relative nef cone by Proposition
  \ref{prop-movablecone}.) We will do this case-by-case below.

  Given this, we complete the proof as follows. If $D$ is a divisor
  which maps into the interior of the cone $\Nef{X_i/\PP^{n-1}}^e$, in
  other words a relatively ample divisor, then by
  \cite[Lemma]{KollarMori1998} the divisor $D-\frac{m}{n-1}K_X$ is
  ample on $X_i$ for $m$ sufficiently large. Since each $X_i$ is
  obtained by flopping fiber components, the class $-\frac{1}{n-1}K_X$
  belongs to $\Nef{X_i}^e$, so this proves that $\Nef{X_i}^e$ surjects
  onto the interior of $\Nef{X_i/\PP^{n-1}}^e$. But now since both these
  cones are rational polyhedral, in fact the image of $\Nef{X_i}^e$ must
  be the whole cone $\Nef{X_i/\PP^{n-1}}^e$.  
\smallskip

So in each case we must identify a finite set of SQMs $X_i$ obtained by flopping
fiber components such that $V$ is covered by the union of the cones
$\Nef{X_i/\PP^n}$.
\begin{enumerate}
\item {\bf $Z=Gr(2,5)$}. Here $k=5$, so the final condition in the
  definition of $V_1$ is $\sum_i a_i \leq 2$. Let $x$ be a point in
  $V_1$: since $a_i \geq 0$ for $i=2,3,4,5$ we have $x \cdot (l-l_i)
  \geq 0$ for these values of $i$. So in order to make the class $x$
  relatively nef we may need to flop classes of the form $l-l_1$,
  $2l-l_1-l_j$ where $\sum_{i \neq j} a_i >1$ and $3l-l_1-l_j-l_k$
  where $\sum_{i \neq j,k} a_i >1$. Note that since $a_i \leq 1$ we
  never need to flop the class of a quartic.

To check that this sequence of flops makes $x$ relatively nef, we need
to check its degree on the classes of the new fiber components
created. As explained in Section \ref{section-flops} these new classes
are of the form $F+(l-l_1)$, $F+(2l-l_1-l_j)$,
$F+(3l-l_1-l_j-l_k)$. But now 
\begin{align*}
x \cdot (F+(l-l_1)) = 2- \sum_i a_i \geq 0
\end{align*}
by the description of $V_1$, and so
\begin{align*}
x \cdot (F+(2l-l_1-l_j)) = 2- \sum_i a_i +a_j \geq 0, \\
x \cdot (F+(3l-l_1-l_j-l_k)) = 2- \sum_i a_i +a_j+a_k \geq 0 \\
\end{align*}
also. Therefore $V_1$ is covered by the relative nef cones of the
finitely many SQMs obtained by flopping some sets of curves with
classes $l-l_1, 2l-l_1-l_j, 3l-l_1-l_j-l_k$. 

\item $Z = Q_1 \cap Q_2 \subset \PP^{n+2}$: Here $k=4$ so we have the
  condition $\sum_i a_i \leq \frac32$. In this case, to make a class $x \in
  V_1$ relatively nef we may have to flop curves of the form $l-l_1$
  and $2l-l_1-l_j$ where  $\sum_{i \neq j} a_i >1$. Since $a_i \leq 1$
  we never need to flop the class of a cubic. 

  Again we must check such a sequence of flops makes $x$ relatively
  nef. In this case the new fiber components have classes $F+(l-l_1)$
  or $F+(2l-l_1-l_j)$, and the same calculation as above shows that
  these numbers are nonnegative.

\item $Z$ is a cubic hypersurface in $\PP^{n+1}$. Here $k=3$ so we have
  the condition $\sum_i a_i \leq 1$: that is, $V_1$ is the unit
  simplex. By Theorem \ref{nef} the projection $\Nef{X}^e
  \arrow V$ is surjective.

\item $Z= \PP^2 \times \PP^2$: Here $k=7$ so we have the condition
  $\sum_i a_i \leq 3$. In this case, the sequence of flops we need to
  perform may be more complicated: we may have to flop two components
  of the same fiber in sequence. In more detail, one checks that
  divisors of the form 
\begin{align*}
x &=S_1 + \sum_{i=2}^7 a_i (S_i -S_1)
\end{align*}
with $0 \leq a_i \leq 1$ and $\sum_i a_i \leq 3$ have intersection
numbers $x \cdot C \geq -4$ with fiber components $C$.  Since $x
\cdot F =1$ by definition of the region $V_1$, one has $x \cdot (4F+C)
\geq 0$ for all $C$ which are components of fibers. By the discussion
in Section \ref{section-flops}, this shows that after an appropriate
sequence of flops the class $x$ becomes relatively nef.

 \item[(5, \,6)] $Z=F(1,2;3)$ or $Z=\PP^1 \times \PP^1 \times \PP^1$. These varieties
  have dimension 3, so Kawamata's proof of the relative version of the
  cone conjecture \cite{Kawamata1997} applied to $f: X \arrow \PP^2$
  shows that in both cases there are finitely many flops $\{X_i\}$ such
  that $V$ is contained in the union of the cones
  $\Nef{X_i/\PP^2}$. For the purposes of the cone conjecture, it is
  not necessary to identify these flops explicitly.

\item[(7)] $Z \arrow \PP^n$ is a double cover branched on a quartic. Here
  $k=2$ so we have the condition $\sum_i a_i \leq \frac12$. That is,
  $V_1$ is the interval $[0,\frac12]$ in the affine line $W$. Again we
  have that $\Nef{X}^e$ surjects onto $V$. (In this case, the
  image of $\Nef{X}^e$ is strictly larger than $V$, due to the extra
  automorphism of $X$ described at the end of Section
  \ref{section-nef}.)

\end{enumerate}
\end{proof}

We can now complete the proof of the cone conjecture.
\begin{theorem}\label{main-corollary}
The second statement of Conjecture \ref{conj-coneconjecture} holds for $X$.
\end{theorem}
\begin{proof} 
  The idea of the proof is to show that given any SQM $X_\alpha$ of
  $X$, there is a pseudo-automorphism $\varphi$ such that $\varphi_*
  \Nef{X_j} = \Nef{X_\alpha}$, where $X_j$ is one of the finitely many
  SQMs $\{X_i\}$ in Proposition \ref{prop-covering}. By Lemma
  \ref{lemma-conestrick}, it suffices to prove this when
  $\Nef{X_\alpha}$ shares a codimension-one face with one of the cones
  $\Nef{X_i}$. We first assume the following claim:
\begin{claim}\label{claim-to-prove}
 Every codimension-one face of a cone $\Nef{X_i}$ which
intersects the interior of the movable cone is dual to the class of a
fiber component.
\end{claim}
The claim implies that if $\Nef{X_\alpha}$ shares a codimension-one
face with $\Nef{X_i}$, then $X_\alpha$ is obtained from $X_i$ by
flopping the class of a fiber component. Since the relative nef cones
$\Nef{X_i/\PP^{n-1}}$ cover the fundamental domain $V$, there exists
some $j$ and some pseudo-automorphism $\varphi$ such that $\varphi_*
\Nef{X_j/\PP^{n-1}} = \Nef{X_\alpha/\PP^{n-1}}$. Choosing an ample
class $A_\alpha$ on $X_\alpha$, the previous Proposition implies there
is an ample class $A_j$ on $X_j$ such that
\begin{align*}
  \varphi_*(A_j) & = A_\alpha + \frac{m}{n-1}K_X
\end{align*}
for some integer $m$. 

If $m \geq 0$, then we can add $-\frac{m}{n-1}K_X$ to both sides of
this equation. Since pseudo-automorphisms preserve $\frac{1}{n-1}K_X$,
this gives
\begin{align*}
  \varphi_*(A_j-\frac{m}{n-1}K_X) & = A_\alpha.
\end{align*}
Moreover, since $\frac{1}{n-1}K_X$ is semi-ample on $X_j$, the class
$A_j-\frac{m}{n-1}K_X$ belongs to $\Nef{X_j}$, and we conclude
that $\varphi_*(\Nef{X_j}) = \Nef{X_\alpha}$.
\smallskip

If $m \leq 0$, then $A_\alpha + \frac{m}{n-1}K_X$ belongs to
$\Nef{X_\alpha}$ provided we know that $-\frac{1}{n-1}K_X$ is nef on
$X_\alpha$. By Claim \ref{claim-to-prove}, $X_\alpha$ is obtained from
some $X_i$ by flopping the class of some fiber component, so
$-\frac{1}{n-1}K_X$ remains nef (indeed base-point-free) on
$X_\alpha$. So again in this case have $\varphi_*(\Nef{X_j}) =
\Nef{X_\alpha}$ as required.

\smallskip

There remains to verify Claim \ref{claim-to-prove}. First, assume that
$\dim X = 3$. By \cite[Theorem 6.15]{KollarMori1998}, the flop of a
smooth threefold is again smooth. Hence, in this case, all $X_i$
occurring in Proposition \ref{prop-covering} are smooth threefolds.
By Mori's theorem \cite{Mori1982}, there are no small $K$-negative
extremal contractions on smooth threefolds, so a codimension-one face
separating two nef cones must be dual to the class of a fiber
component. 

Next, let us verify Claim \ref{claim-to-prove} for the case $Z= \PP^2
\times \PP^2$. We wish to prove that there are no small
$K$-negative extremal contractions of any of the flops $X_i$. First,
we note that each $X_i$ is smooth. To see this, note that any sequence
of flops of $X$ restricts to a sequence of flops of any smooth
threefold $Y \in |H|$ inside $X$. As explained in the previous
paragraph, any sequence of flops applied to $Y$ yields a smooth
threefold $Y_i$. If now the resulting fourfold $X_i$ were singular,
there would be a smooth threefold $Y_i \subset X_i$ which is the
support of a Cartier divisor, and which intersects the singular locus
of $X_i$. This is impossible, so we conclude that $X_i$ must be
smooth.

Now let us show that there are no small $K$-negative extremal
contractions of $X_i$. By the Cone Theorem \cite{KollarMori1998}, any
$K$-negative extremal ray of the cone of curves $\Curv{X_i}$ is
spanned by a rational curve $C$ with $0<-K_X \cdot C \leq \dim
X+1=5$. Since $-K_X=3H$, the only possibility is that $H \cdot
C=1$. By basic deformation theory \cite[I.2.17]{Kol}, at the point
corresponding to the inclusion of the curve $C$, the space of maps
$\PP^1 \arrow X$ modulo isomorphisms of the map has dimension at least
$-K_X \cdot C + \dim{X} - 3 =4$. Now suppose that the deformations of
$C$ swept out only a surface $S$ inside $X_i$. Then the space of
curves through 2 general points of $S$ must have dimension at least
2. By Bend and Break \cite[II.5]{Kol}, there must be a deformation of
$C$ which is reducible. But $C$ is a generator of an extremal ray of
the cone curves $\Curv{X_i}$, and is primitive since $H \cdot C =1$,
so this is impossible. We conclude that deformations of $C$ must cover
a locus of dimension at least 3 in $X_i$, as required.

To verify Claim \ref{claim-to-prove} in the remaining cases, we
need to find all the codimension-one faces of the cones $\Nef{X_i}$
which intersect the interior of the movable cone. As usual, we will do
this case-by-case.

\begin{enumerate}
\item {\bf $Z=Gr(2,5)$.} As explained above, here the SQMs $X_i$ are
  obtained by flopping classes of the following kinds:
\begin{align*}
  l-l_1, \, 2l-l_1-l_i (i=2,\ldots,5), \, 3l-l_1-l_i-l_j
  (i,j=2,\ldots,5, i \neq j).
\end{align*}
Note that there are restrictions on which sets of curves can be
flopped: the class $l-l_1$ must be flopped first, cubics
$3l-l_1-l_i-l_j$ cannot be flopped unless conics through $p_i$ and
$p_j$ have already been flopped, and two cubics whose union passes
through all points cannot both be flopped, because of the condition
$\sum_i a_i \leq 2$ in the definition of $V_1$. The following is a complete
list (up to relabeling) of all possible sequences of flops of classes
of these kinds; for brevity, we use the notation $1 \cdots i$ to
indicate the flop of the class $kl-l_1-\cdots-l_i$.
\begin{align*}
  &(1) \\
  &(1, \, 12), \quad (1,\, 12, \, 13),\quad (1,\, 12, \, 13, \, 14),
  \quad
  (1,\, 12, \, 13, \, 14, \, 15),  \\
  &(1,\, 12, \, 13, \, 123), \quad (1,\, 12, \, 13, 14, \, 123) \quad
  (1,\, 12,
  \, 13, \,14, \, 15, \, 123) \\
  &(1,\, 12, \, 13, \, 14, \, 123, \, 124), (1,\, 12, \, 13, \, 14, \,
  15, 123, \, 124) \\\ &(1,\, 12, \, 13, \, 14, \, 123, \, 124, \,
  134) \quad (1,\, 12, \, 13, \, 14, 15, \, 123, \, 124, \, 134) \\
&(1,\, 12, \, 13, \, 14, \, 15, \, 123, \, 124, \, 125) \quad 
\end{align*}
It is straightforward to calculate the nef cones of all these SQMs and
observe that their codimension-one faces which intersect the interior
on the movable cone are dual to fiber components, as required. Here is
a list of the extremal rays of the nef cone in each case on the above
list:
\begin{itemize}
\item $(1)$: $\left\{ H-E_1 - \sum_i E_i\right\}$, $H-2E_1$.
\item flops involving only line and conics, e.g. $ (1, \, 12)$: here
  the extremal rays are of the form $\left\{ H-E_1 - E_2 - \sum_i
    E_i\right\}$, \,
  $H-2E_1$, \,$H-2E_1-E_2$.
\item flops involving lines, conics, and one cubic, e.g. $(1, \, 12, \,
  13, \, 123)$: here the extremal rays are of the form \newline

  $\left\{ H-E_1 - E_2 - E_3 - \sum_i E_i\right\}$, \,$H-2E_1-E_2$,
  \,$H-2E_1-E_3$, \,$2H-3E_1-2E_2-2E_3-E_i \, (i=4,5)$,
  \,$2H-3E_1-2E_2-2E_3-E_4-E_5$.
\item flops involving lines, conics, and more than one cubic,
  for example,  $(1, \, 12, \, 13, \, 14, \, 123, \, 124)$: here the extremal
  rays are of the form \newline
  $\left\{ H-E_1 - E_2 - E_3 - E_4 - \sum_i E_i\right\}$,
  \,$H-2E_1-E_2$, $H-2E_1-E_3$, \,$H-2E_1-E_4$, \,$2H-4E_1-E_2-E_3-E_4$,
  \,$2H-4E_1-E_2-E_3-E_4-E_5$, \,$2H-3E_1-2E_2-2E_3-E_i \, (i=4,5)$,
  \,$2H-3E_1-2E_2-2E_3-E_4-E_5$, \,$3H-5E_1-3E_2-2E_3-2E_4$,
  \,$3H-5E_1-3E_2-2E_3-2E_4-E_5$, \,$4H-7E_1-3E_2-3E_3-3E_4$,
  \,$4H-7E_1-3E_2-3E_3-3E_4-2E_5$.

\end{itemize}
Let us verify that these classes are indeed nef on the appropriate
SQMs:
\begin{itemize}
\item Classes in $\Nef{X}$: we saw in Section \ref{section-nef} that
  all nef divisors on $X$ are effective. A class of this type only
  appears in a nef cone where we have flopped curves which are
  disjoint from some representative of that class. So it remains nef
  on such a class.
\item The class $H-2E_1$: this class is represented by any
  hyperplane section of $Gr(2,5)$ which contains $T_{p_1} Gr(2,5)$, so
  the base locus of this class is $T_{p_1} \cap Gr(2,5)$, which is the
  locus of lines on the Grassmannian through $p_1$, that is the locus
  of curves with class $l-l_1$. Flopping this locus, this class
  becomes base-point-free, in particular nef, and remains so if we flop
  loci of curves which are disjoint from some representative of the class.

\item Classes of the form $H-2E_1-E_i$: such a class is represented by
  a hyperplane section of $Gr(2,5)$ which contains $T_{p_1} Gr(2,5)$
  and the point $p_i$. The linear space $T_{p_1}$ and the point $p_i$
  together span a subspace $\PP^7 \subset \PP^9$, so there is a
  1-parameter family of such hyperplane sections. The base locus of
  this family certainly contains the locus of curves with classes
  $l-l_1$ and $2l-l_1-l_i$: on the other hand, by Pieri's rule
  \cite[I.5]{gh}, it has class $\sigma_1^2= \sigma_{1,1}+ \sigma_2$,
  so it must be equal to this locus. Flopping these loci, this class
  becomes base-point-free, in particular nef, and again remains so on
  further flops.
\item Classes of the form $2H-3E_1-2E_2-2E_3-E_i$: we can decompose
  such a class into effective classes in different ways, as follows:
\begin{align*}
H-2E_1-E_2-E_4 &+ H-E_1-E_2-2E_3, \\
H-2E_1-E_3-E_4 &+H-E_1-2E_2-E_3.
\end{align*}
(Here each class is effective since there is a unique $\PP^8 \subset
\PP^9$ containing the tangent space $T_{p_1}X$ and two other general
points.) These decompositions show that the base locus of this class
on $X$ is exactly the locus of curves $l-l_1$, $2l-l_1-l_i$,
$3l-l_1-l_i-l_j$. Moreover, these hyperplanes have distinct normal
directions along the curves to be flopped, so after flopping all the
curves in the base locus, the proper transforms of the hyperplanes
become disjoint. Therefore, this class becomes base-point-free after
performing the appropriate sequence of flops.
\item The class $D=2H-3E_1-2E_2-2E_3-E_4-E_5$: First, one can
  see that this class is effective by writing it in the form 
\begin{align*}
-\frac{1}{5}K_X+ (H-2E_1-E_2-E_3).
\end{align*}
To show this class is nef on the appropriate flop is somewhat more
complicated. Applying the element $\varphi=(E_1-E_4) \in$ MW$(f)$ to our
class transforms it to a class of the form $-\frac{m}{5}K_X +
(H-E_1-E_2-E_3-E_4)$ for some integer $m$. By restricting to the
preimage $S$ of a general line in $\PP^5$ and using invariance of
intersection numbers under automorphisms, we find that $m=0$: that is,
$\varphi(D)=D'=H-E_1-E_2-E_3-E_4$. We saw in Section \ref{section-nef}
that the class $D'$ is semi-ample on $X$; replacing $D$ and $D'$ by
appropriate multiples, we can assume it is base-point-free. Since
$\varphi$ and hence $\varphi^{-1}$ are elements of MW$(f) \subset
\PsAut(X/\PP^5)$, it follows that the base locus of $D$ is a union of
fiber components. If we perform any sequence of flops of fiber
components and take the proper transform the same is true: that is, on
any SQM of $X$ obtained by flopping fiber components, the proper
transform of $D$ is nonnegative on any curve that is not a fiber
component. On the other hand, we know that after performing the
appropriate sequence of flops, $D$ becomes relatively nef, i.e. nef on
all fiber components. Hence, $D$ is nef on the appropriate SQM of $X$.
\item Classes of the form $2H-4E_1-E_2-E_3-E_4$ or
  $2H-4E_1-E_2-E_3-E_4-E_5$: the second class decomposes into a sum of
  effective divisors in two different ways as follows:
  $(H-2E_1-E_2-E_3)+(H-2E_1-E_ 4-E_5)$, $(H-2E_1-E_2-E_4)+(H-2E_1-E_
  3-E_5)$. Since these are all distinct irreducible prime divisors,
  the base locus of this class has codimension 2. Moreover, one can see
  that the intersection of these two representatives of the class is
  exactly the locus to be flopped. A similar (easier) argument works
  for the first class.
\item Classes of the form $3H-5E_1-2E_3-2E_4$ or
  $3H-5E_1-2E_3-2E_4-E_5$: as in a previous case, there are
  pseudo-automorphisms of $X$ taking these classes to $H-E_1-E_2$ or
  $H-E_1-E_2-E_5$. The same argument as before allows us to conclude
  that these classes are nef on the appropriate flops. 
\item Classes of the form $4H-7E_1-3E_2-3E_3-3E_4$ or
  $4H-7E_1-3E_2-3E_3-3E_4-2E_5$: again there are pseudo-automorphisms
  taking these classes to $H-E_5$ and $H-2E_1-E_5$ respectively. The
  first class is base-point-free on $X$, and the base locus of the
  second is known. We can then argue as before to conclude that these
  classes are nef on the appropriate flop. 
  
\end{itemize}

\item {\bf $Z$ is the intersection of two quadrics in $\PP^{n+2}$.}
  Here the SQMs $X_i$ are obtained by flopping classes of the
  following kinds:
\begin{align*}
l-l_1, \, 2l-l_1-l_i (i=2,\ldots,4).
\end{align*}
Here it is possible to perform any sequence of flops starting with the
flop of $l-l_1$. The calculation of the nef cones of the resulting
SQMs is formally identical to the previous case, so again we get the
nef cones
\begin{itemize}
\item $(1)$: $\left\{ H-E_1 - \sum_i E_i\right\}$, \,$H-2E_1$.
\item $ (1, \, 12)$: $\left\{ H-E_1 - E_2 - \sum_i E_i\right\}$,
  \,$H-2E_1$, \,$2H-3E_1-2E_2$, \,$2H-3E_1-2E_2-E_i$, \,$2H-3E_1-2E_2-E_3-E_4$.
\item $(1, \, 12, \, 13)$: $\left\{ H-E_1 - E_2 - E_3 - \sum_i
    E_i\right\}$, \,$H-2E_1$, \,$2H-3E_1-2E_2-E_3$,
  \,$2H-3E_1-2E_2-E_3-E_4$, \,$2H-3E_1-2E_3$, \,$3H-5E_1-2E_2-2E_3$,
  \,$3H-5E_1-2E_2-2E_3-E_4$.

\item $(1, \, 12, \, 13, \, 14)$: $H-E_1 - E_2 - E_3 -E_4$, \,$H-2E_1$,
  $2H-3E_1-2E_i-E_j-E_k \, (i,j,k=2,3,4)$, \,$3H-5E_1-2E_i-2E_j-E_k \,
  (i,j,k=2,3,4)$, \,$4H-6E_1-3E_2-3E_3-3E_4$, \,$5H-8E_1-4E_i-3E_j-3E_k \,
  (i,j,k=2,3,4)$, \,$6H-10E_1-4 E_i-4 E_j-3E_k \, (i,j,k=2,3,4)$,
  $7H-12E_1-4E_2 -4E_3- 4E_4$.
\end{itemize}
\begin{itemize}
\item Classes in $\Nef{X}$ are effective; those with representatives
  disjoint from flopping locus remain nef on the flops.
\item The class $H-2E_1$: again this is represented by any hyperplane
  section of $Z$ which contains $T_{p_1}$. Again the base locus is
  $T_{p_1} \cap Z$, which is the locus of lines through $Z$. The
  argument from the previous case goes through unchanged.
\item Classes $2H-3E_1-2E_i$: this decomposes into effective divisors
  as $(H-E_1-2E_2)+(H-2E_1)$ and $(H-2E_1-E_2)+(H-E_1-E_2)$. In both
  decompositions the second term is a movable class, and the fixed
  parts of the two decompositions intersect properly, so this class is
  movable. Moreover, one can read off the base locus from the
  decompositions and see that the class becomes nef on the appropriate
  flop. 
\item Classes $2H-3E_1-2E_i-E_j$: again we have two decompositions not
  sharing a divisor, namely $(H-2E_1-E_3)+(H-E_1-2E_2)$ and
  $(H-2E_1-E_2)+(H-E_1-E_2-E_3)$. The standard argument shows these
  classes become nef on the appropriate flop. 
\item Classes $2H-3E_1-2E_i-E_j-E_k$: there is a pseudo-automorphism
  taking this class to the class $H-E_1-E_i-E_j$, which is
  base-point-free, so we can argue as before to conclude that it is nef
  on the appropriate flop. 
\item Classes $3H-5E_1-2E_i-2E_j$, $3H-5E_1-2E_i-2E_j-E_k$: there are
  pseudo-automorphisms taking these classes to $H-2E_1$ and $H-E_1-E_2$,
  respectively. The latter two classes are base-point-free on $X$, so
  we argue as before. 
\item The class $4H-6E_1-3E_2-3E_3-3E_4$: there is a
  pseudo-automorphism taking this class to the base-point-free class
  $H-E_1-E_2-E_3$.
\item The class  $5H-8E_1-4E_i-3E_j-3E_k$: there is a
  pseudo-automorphism taking this to the base-point-free class
  $H-E_1-E_2$. 
\item The class $6H-10E_1-4 E_i-4 E_j-3E_k$: there is a
  pseudo-automorphism taking this to the class $2H-3E_1-2E_2$. We have
  already shown that the latter class is movable and identified its
  base locus, so we can use the same arguments as before to conclude
  that the original class becomes nef on the appropriate flop.
\item The class $7H-12E_1-4E_2 -4E_3- 4E_4$: first we apply a
  pseudo-automorphism to take this to the class $3H-4E_1-4E_2$. The
  latter class can be split up into effective classes in two ways, as
  $(2H-3E_1-2E_2)+(H-E_1-2E_2)$ and
  $(2H-2E_1-3E_2)+(H-2E_1-E_2)$. The first term in each is movable, as
  already shown, and the two fixed divisors intersect in codimension
  2. We can then argue as before to conclude that this class becomes
  nef on the appropriate flop.
\end{itemize}

\item {\bf $Z$ is a cubic hypersurface in $\PP^{n+1}$.} In this case
  the set $\{X_i\}$ consists of a single SQM, namely $X$ itself. The
  nef cone $\Nef{X}$ is described in Theorem \ref{nef}, and one reads
  off that the faces which intersect the interior of the movable cone
  are dual to fiber components.

\item {\bf $Z \arrow \PP^n$ is a double cover branched over a
    quartic.} Again in this case the set $\{X_i\}$ consists only of
  $X$ itself, and Theorem \ref{nef} tells us that the faces of the
  cone of curves which intersect the interior of the movable cone are
  dual to fiber components. 
\end{enumerate}

This completes the proof of Claim \ref{claim-to-prove}. As explained
above, this shows that for each of our examples $X$, there is a finite
collection $\{X_i\}$ of SQMs of $X$ such that each nef cone
$\Nef{X_i}^e$ is rational polyhedral and spanned by effective
divisors, and such that 
\begin{align*}
\PsAut^*(X/\PP^{n-1}) \cdot \left ( \bigcup_i \Nef{X_i} \right) =
\Mov{X}^e.
\end{align*}

Given this, it is then straightforward to produce a precise rational
polyhedral fundamental domain for the action of the larger group
$\PsAut^*(X,\Delta)$ on $\Mov{X}^e$, as required.
\end{proof}

\section{Appendix: Sextic hypersurface in weighted projective
  space} \label{appendix} In this appendix, we show that the
Morrison--Kawamata conjecture also holds for case (5) on the list of
Fano manifolds of index $n-1$ in Section
\ref{section-introduction}. The proof here is rather simpler than in
the other cases, since it turns out the the movable cone is rational
polyhedral. For facts about weighted projective spaces, see for
example \cite{Fletcher}.
\begin{proposition}
  Let $Z$ be a general hypersurface of degree 6 in the weighted
  projective space $\PP(3,2,1,\ldots,1)$ of dimension $n+1$. Then
  there is a line bundle $H$ on $Z$ such that $-K_Z=(n-1)H$ and
  $H^n=1$. Let $X$ be the blowup of $Z$ in the point $p$, defined as
  the base locus of $H$. Then Conjecture \ref{conj-coneconjecture}
  holds for $X$.
\end{proposition}
\begin{proof}
  The canonical class of the weighted projective space
  $\PP=\PP(a_0,\ldots,a_{n+1})$ (where
  $\operatorname{gcd}(a_0,\ldots,a_{n+1})=1$) is $K_\PP = \OO(-\sum
  a_i)$. So by adjunction a hypersurface $Z$ of degree $d$ in this
  space has canonical class $K_Z=\OO(d-\sum_{i=0}^{n+1}
  a_i)$. Applying this in our case where $n$ of the weights equal $1$,
  we get $-K_Z=\OO(-6+3+2+n)=\OO(n-1)$. So $H=\OO(1)_{|Z}$. By the
  weighted form of B\'ezout's Theorem, we get $H^n=\frac{1}{3 \cdot 2
    \cdot 1 \cdots 1} \cdot 6 \cdot 1^n = 1$, as claimed.

  Now let $X$ be the blowup of $Z$ in the base locus of $H$. Note that
  $\rho(X)=2$. As before we get an elliptic fibration $f:X \arrow
  \PP^{n-1}$. To prove the first statement of the cone conjecture for
  $X$, it suffices to exhibit two contractions of $X$ which contract
  different curves. To see that this is enough, note that two line
  bundles $L_1$, $L_2$ corresponding to the contractions must span the
  two edges of $\Nef{X}$; moreover $L_1$ and $L_2$ must be semiample,
  in particular effective, so $\Nef{X}^e=\Nef{X}$, a rational
  polyhedral cone. The contractions we need are the blow-down map
  $\pi: X \arrow Z$ and the elliptic fibration $f: X \arrow
  \PP^{n-1}$.

  The statement about the movable cone follows immediately: in fact
  $\Mov{X}^e=\Nef{X}^e$. To see this, observe that by the previous
  paragraph, the cone of curves $\Curv{X}$ is spanned by $C_1$, the
  class of a line in the exceptional divisor of $\pi$, and $C_2$, the
  class of a fiber of $f$. So a non-nef divisor on $X$ must have
  negative degree on either $C_1$ or $C_2$. Since curves in each of
  these classes fill up a locus of codimension $\leq 1$ in $X$, no
  such divisor can be movable.
\end{proof}

\end{document}